\documentclass[11pt]{amsart}
\usepackage[utf8]{inputenc}
\usepackage{kantlipsum}
\usepackage[pagebackref=true, colorlinks=true,linkcolor=blue, citecolor={red}]{hyperref}
\usepackage{lmodern}\usepackage[T1]{fontenc}
\usepackage{amsmath,amssymb,amsfonts,euscript,graphicx,hyperref,indentfirst}
\usepackage{enumerate}
\usepackage[all]{xy}
\usepackage{tikz-cd}
\usepackage{stmaryrd}
\usepackage{relsize,exscale}
\calclayout

\newtheorem{theorem}{Theorem}[section]
\newtheorem{lemma}[theorem]{Lemma}
\newtheorem{proposition}[theorem]{Proposition}

\newtheorem{convention-proposition}[theorem]{Convention-Proposition}

\theoremstyle{definition}

\newtheorem{example}[theorem]{Example}
\newtheorem{convention}[theorem]{Convention}

\theoremstyle{remark}
\newtheorem{remark}[theorem]{Remark}

\numberwithin{equation}{section}

\title{On the Non-monotonicity of entropy for a class of real quadratic rational maps}
\author{Khashayar Filom and Kevin M. Pilgrim}
\date{June 2020}

\address{Khashayar Filom, Department of Mathematics, Northwestern University;
}
\email{khashayarfilom2014@u.northwestern.edu}

\address{Kevin M. Pilgrim, Department of Mathematics, Indiana University;
}
\email{pilgrim@indiana.edu}

\begin{document}

\begin{abstract}
We prove that the entropy function on the moduli space of real quadratic rational maps is not monotonic by exhibiting a continuum of disconnected level sets. This entropy behavior  is in stark contrast with the case of polynomial maps, and establishes a conjecture on the failure of monotonicity for bimodal real quadratic rational maps of shape $(+-+)$ which was posed in \cite{2019arXiv190103458F} based on experimental evidence.  
\end{abstract}

\maketitle

\section{Introduction}\label{intro}
The variation of entropy in a family of dynamical systems is a natural indication of the change of dynamics through the family that could shed light on the nature of bifurcations.  There is a vast literature on the entropy behavior of interval maps. In particular, Milnor's conjecture on the \textit{monotonicity of entropy} claims that the entropy level sets -- the \textit{isentropes} --  are connected within families of polynomial interval maps whose critical points are all real \cite{MR3289915}. The monotonicity of entropy for polynomial interval maps was first established for quadratic polynomials \cite{MR762431,MR970571,MR1351519}. In case of the logistic family 
\begin{equation}\label{quadratic family}
\left\{x\mapsto\mu x(1-x): [0,1] \to [0,1]\right\}_{0\leq \mu\leq 4},
\end{equation}
this monotonicity result states that the entropy is non-decreasing with respect to the parameter $\mu$. Next, the monotonicity conjecture was settled in the case of cubic polynomials in 
\cite{MR1351522,MR1736945}. In the general setting of boundary-anchored polynomial interval maps of a fixed degree and shape and with real non-degenerate critical points, the monotonicity of entropy has been established in \cite{MR3264762}, and also in \cite{MR3999686} with a different method.  

In this paper we are concerned with rational maps rather than polynomials. After some effort one can still set up an entropy function on an appropriate moduli space of real rational maps \cite{2018arXiv180304082F}. When the degree is two, the space $\mathcal{M}_2(\Bbb{R})$ of Möbius  conjugacy classes of real quadratic rational maps may be naturally identified with the real plane $\Bbb{R}^2$ \cite[\S10]{MR1246482}. Assigning to each conjugacy class $\langle f\rangle$ the topological entropy of the restriction to the real circle $\hat{\Bbb{R}}:=\Bbb{R}\cup\{\infty\}$ of a representative $f$ defines a continuous \textit{real entropy function} 
\begin{equation}\label{real entropy 1}
h_{\Bbb{R}}:\langle f\rangle\mapsto h_{\rm{top}}\left(f\restriction_{\hat{\Bbb{R}}}\hat{\Bbb{R}}\rightarrow\hat{\Bbb{R}}\right)\in\left[0,\log(2)\right]
\end{equation}
on an appropriate open subset of $\mathcal{M}_2(\Bbb{R})$. The numerically generated entropy contour plots in \cite{2019arXiv190103458F} suggested that the isentropes are connected in certain dynamically defined regions of the moduli space whereas are disconnected in another region of dynamical interest; namely, the region of \textit{$(+-+)$-bimodal}  maps. The former was partially resolved in that paper (\cite[Theorem 1.2]{2019arXiv190103458F}) while the non-monotonicity part was stated merely as a conjecture (\cite[Conjecture 1.4]{2019arXiv190103458F}). The main goal of this paper is to establish this anticipated failure of monotonicity; see Theorem \ref{main} below. We prove the non-monotonicity of the real entropy function $h_\Bbb{R}$ by studying certain $(+-+)$-bimodal real quadratic rational maps. Our arguments easily imply the non-monotonicity for the restriction of $h_\Bbb{R}$ to the  $(+-+)$-bimodal region as well.  
\begin{theorem}\label{main}
There exists a real number $h'\in\left(0,\log(2)\right)$ with the property that for every entropy value $h\in\left(h',\log(2)\right)$ the level set $h_\Bbb{R}=h$ is disconnected.
\end{theorem}
\noindent 
In fact, we can take $h'$  to be the logarithm of the largest real root of $t^3-2t^2+1=0$:
\begin{equation}\label{h'}
h'=\log\left(\frac{1+\sqrt{5}}{2}\right).
\end{equation}

The main idea of the proof is to construct certain unbounded hyperbolic components -- denoted by $\mathcal{H}_{p/q}$ -- of $\mathcal{M}_2(\Bbb{R})$, and then to use the elementary fact that the entropy remains constant throughout any real hyperbolic component. After a brief review based on \cite{MR1047139,MR1246482} of the background material on the moduli space of quadratic rational maps and its hyperbolic components in \S\ref{background}, we construct such unbounded hyperbolic components  in \S\ref{construction}. The main ingredient of the construction is to exhibit certain \textit{post-critically finite} (\textit{PCF} for short) real hyperbolic rational maps $f_{p/q}$ with a specified dynamics on $\hat{\Bbb{R}}$ which lie at the \textit{center} of the aforementioned hyperbolic components $\mathcal{H}_{p/q}$. The construction, a special case of \cite{MR1609463}, is topological and  utilizes  \textit{Thurston's characterization of rational maps} \cite{MR1251582}. 
Next, we proceed in \S\ref{construction} with an analysis of the limit points of $\mathcal{H}_{p/q}$ in a compactification of $\mathcal{M}_2(\Bbb{R})\cong\Bbb{R}^2$ to a closed disk. The analysis of the degeneration of  hyperbolic components $\mathcal{H}_{p/q}$  is reminiscent of ideas developed in \cite{MR2691488,MR1914001}, and also relies on \cite{MR1257034}.   The proof of the main theorem finally appears in \S\ref{proof}, and utilizes the properties of components  $\mathcal{H}_{p/q}$ discussed in the previous section, an entropy monotonicity result of Levin, Shen, and van Strien  \cite[Theorem 7.2]{2019arXiv190206732L}, and planar topology arguments. 

\textbf{Acknowledgments.} K.F. is grateful to Laura DeMarco for helpful conversations and suggestions, to Yan Gao for introducing him to paper \cite{2019arXiv190206732L}, and to the department of mathematics at Indiana University Bloomington for its hospitality during visits in September 2018 and May 2019. K.M.P. was supported by Indiana University Bloomington and Simons Foundation collaboration grant No. 245269.

\section{Background on the Moduli Space and Hyperbolic Components of Quadratic Rational Maps}\label{background}
The goal of this section is to present a brief account of the moduli space of (real or complex) quadratic rational maps including the dynamical coordinate system that identifies the moduli space with a plane, the corresponding compactifications, the seven different topological types of a real quadratic rational map, the experimental evidence on which Theorem \ref{main} is based; and finally, the hyperbolic components of quadratic rational maps which are vital to the proof of the theorem.    

The complex moduli space $\mathcal{M}_2(\Bbb{C})$ of quadratic rational maps is defined as the space
$${\rm{Rat}}_2(\Bbb{C})\big/{\rm{PSL}}_2(\Bbb{C})=\left\{f \mid f \text{ a rational map of degree two}\right\}\big/f\sim\alpha\circ f\circ\alpha^{-1}$$  
of Möbius conjugacy classes of rational maps
 $f:\hat{\Bbb{C}}\rightarrow\hat{\Bbb{C}}$ of degree two. The conjugacy class of $f$ is denoted by   
 $\langle f\rangle\in\mathcal{M}_2(\Bbb{C})$. This space could famously be identified with the plane $\Bbb{C}^2$  \cite{MR1246482}. To elaborate, recall that such a map $f$ has three fixed points (counted with multiplicity) whose multipliers -- denoted by $\mu_1$, $\mu_2$ and $\mu_3$ -- are related by the \textit{holomorphic fixed point formula}
(see \cite[\S12]{MR2193309})
\begin{equation}\label{fixed point formula}
\frac{1}{1-\mu_1}+\frac{1}{1-\mu_2}+\frac{1}{1-\mu_3}=1.
\end{equation} 
This amounts to a constraint on the symmetric functions 
\begin{equation}\label{symmetric functions}
\sigma_1=\mu_1+\mu_2+\mu_3,\quad \sigma_2=\mu_1\mu_2+\mu_2\mu_3+\mu_3\mu_1,\quad \sigma_3=\mu_1\mu_2\mu_3,
\end{equation}
of these multipliers given by $\sigma_3=\sigma_1-2$. The conjugacy-invariant functions $\sigma_1$ and $\sigma_2$  then identify $\mathcal{M}_2(\Bbb{C})$ with $\Bbb{C}^2$.

As we are concerned with the dynamics on the real line union infinity $\hat{\Bbb{R}}$, only the conjugacy classes of real quadratic rational maps $f\in{\rm{Rat}}_2(\Bbb{R})$ under the action of ${\rm{PGL}}_2(\Bbb{R})$ are relevant to our discussion.  The functions $\sigma_i$'s may be described in terms of coefficients of the rational map $f$. For instance, in the \textit{mixed normal form}
\begin{equation}\label{mixed normal form}
\frac{1}{\mu}\left(z+\frac{1}{z}\right)+a
\end{equation}
where the critical points and a fixed point are specified, $\sigma_1$ and $\sigma_2$ are given by the  formulas below 
adapted from \cite[Appendix C]{MR1246482}:
\begin{equation}\label{coordinates}
\begin{cases}
\sigma_1=\mu(1-a^2)-2+\frac{4}{\mu}\\
\sigma_2=\left(\mu+\frac{1}{\mu}\right)\sigma_1-\left(\mu^2+\frac{2}{\mu}\right)
\end{cases}.
\end{equation}
Therefore, $\sigma_1$ and $\sigma_2$ are real once $f$ lies in ${\rm{Rat}}_2(\Bbb{R})$.  Conversely, any point of $\mathcal{M}_2(\Bbb{C})\cong\Bbb{C}^2$ with real coordinates can be represented by a real map: If the multiplier $\mu$ in \eqref{coordinates} is real (that is, a real root of  the real cubic $z^3-\sigma_1z^2+\sigma_2z-\sigma_3=0$), the real-ness of $\sigma_1$ and $\sigma_2$ requires $a$ to be either real, or purely imaginary of the form ${\rm{i}}b$ with $b$ real. In the latter situation, after a conjugation with $z\mapsto\frac{z}{{\rm{i}}}$ we arrive at a real map of the form 
\begin{equation}\label{mixed normal form-alternative}
\frac{1}{\mu}\left(z-\frac{1}{z}\right)+b.
\end{equation}
We deduce that the space 
$$
\mathcal{M}_2(\Bbb{R})={\rm{Rat}}_2(\Bbb{R})\big/{\rm{PSL}}_2(\Bbb{R})
$$
of the conjugacy classes of real maps could be identified with the underlying real plane $\Bbb{R}^2$ \cite[\S10]{MR1246482}. 
Figure \ref{fig:main} adapted from Milnor's paper illustrates the moduli space $\mathcal{M}_2(\Bbb{R})$ in the 
$(\sigma_1,\sigma_2)$ coordinate system.
The paper then proceeds with a careful examination of $\mathcal{M}_2(\Bbb{R})$  based on the real dynamics which we shall review below:
\begin{itemize}
\item The restriction $f\restriction_{\hat{\Bbb{R}}}:\hat{\Bbb{R}}\rightarrow\hat{\Bbb{R}}$ is either a two-sheeted covering map or is not surjective in which case both critical points of $f$ are real, and the topological degree of $f\restriction_{\hat{\Bbb{R}}}:\hat{\Bbb{R}}\rightarrow\hat{\Bbb{R}}$ is zero \cite[Proposition 2.4]{2019arXiv190103458F}.
\item In the case of topological degree zero, the image of $f$ is a compact interval $f(\hat{\Bbb{R}})$. From the dynamical standpoint, one can solely concentrate on the interval map
$$
f\restriction_{f(\hat{\Bbb{R}})}:f(\hat{\Bbb{R}})\rightarrow f(\hat{\Bbb{R}}).
$$
Conditioning on its modality (how many critical points lie in $f(\hat{\Bbb{R}})$) and shape (whether it starts with an increase or a decrease), one obtains the smaller \textit{monotone increasing}, \textit{monotone decreasing}, \textit{unimodal}, $(+-+)$-\textit{bimodal} and $(-+-)$-\textit{bimodal} regions within the non-covering part of $\mathcal{M}_2(\Bbb{R})$. These along with  \textit{degree} $\pm2$ regions comprise the seven regions partitioning the moduli space $\mathcal{M}_2(\Bbb{R})$ in Figure \ref{fig:main}.
\item The degree $\pm 2$ regions are separated from the union of the other five regions -- which we call the  \textit{component of degree zero maps} -- via the real part of the \textit{symmetry locus}  
$$
\mathcal{S}(\Bbb{R})=\left\{f\in {\rm{Rat}}_2(\Bbb{R}) \mid f \text{ has a non-trivial  Möbius automorphism} \right\};
$$
see Figure \ref{fig:maincolored}. 
The symmetry locus is thoroughly studied in \cite[\S5]{MR1246482}. The real entropy \eqref{real entropy 1} is multi-valued at the points of $\mathcal{S}(\Bbb{R})$ as they represent maps of the form $\frac{1}{\mu}(z\pm\frac{1}{z})$ which are conjugate only via non-real Möbius transformations and thus, restrict to dynamically distinct self-maps of $\hat{\Bbb{R}}$: The map $x\mapsto\frac{1}{\mu}\left(x-\frac{1}{x}\right)$ is covering and hence of entropy $\log(2)$ whereas  $x\mapsto\frac{1}{\mu}\left(x+\frac{1}{x}\right)$ is of entropy zero \cite[Example 2.2]{2019arXiv190103458F}.
\end{itemize}

\begin{figure}[ht!]
\center
\includegraphics[width=13cm]{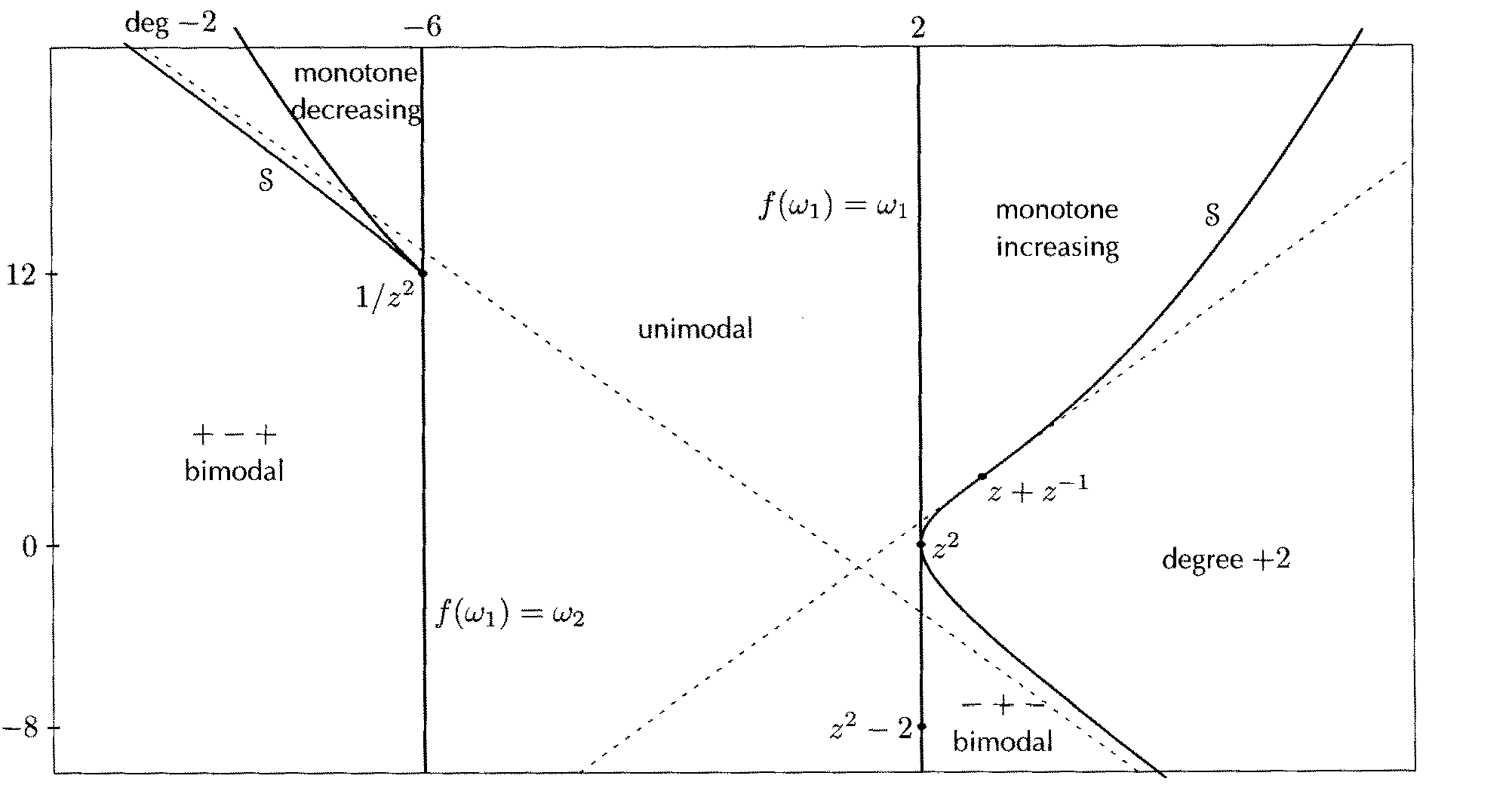}
\caption{The real moduli space $\mathcal{M}_2(\Bbb{R})$ as illustrated in  \cite[Figure 15]{MR1246482}. The post-critical lines $\sigma_1=-6,2$ ($\omega_1,\omega_2$ denote the critical points of $f$ here), the dotted lines ${\rm{Per}}_1(\pm 1)$, the real symmetry locus $\mathcal{S}(\Bbb{R})$ and the partition into seven regions according to the various types of the dynamics induced on $\hat{\Bbb{R}}$ are shown. The component of degree zero maps in 
$\mathcal{M}_2(\Bbb{R})-\mathcal{S}(\Bbb{R})$ is the union of monotonic, unimodal and bimodal regions that overlap only along the lines $\sigma_1=-6,2$.}
\label{fig:main}
\end{figure}

After excluding the symmetry locus, the real entropy \eqref{real entropy 1} gives rise to a single-valued function 
\begin{equation}\label{real entropy function}
\begin{cases}
h_\Bbb{R}:\mathcal{M}_2(\Bbb{R})-\mathcal{S}(\Bbb{R})\rightarrow\left[0,\log(2)\right]\\
h_{\Bbb{R}}(\langle f\rangle):=h_{\rm{top}}\left(f\restriction_{\hat{\Bbb{R}}}:\hat{\Bbb{R}}\rightarrow\hat{\Bbb{R}}\right)=
h_{\rm{top}}\left(f\restriction_{f(\hat{\Bbb{R}})}:f(\hat{\Bbb{R}})\rightarrow f(\hat{\Bbb{R}})\right)
\end{cases}
\end{equation}
which is continuous \cite{MR1372979}. The domain of definition has three connected components; see Figure \ref{fig:maincolored}. In our treatment of the monotonicity problem, only unimodal and bimodal regions matter as $h_{\Bbb{R}}\equiv\log(2)$ over the degree $\pm 2$ components of the domain while $h_{\Bbb{R}}\equiv 0$ over the monotonic regions of the component of degree zero maps. A rather lengthy analysis of the dynamics in the degree zero case reduces everything to the study of certain families of unimodal and bimodal interval maps for which entropy plots could be generated numerically \cite[\S\S 4,5]{2019arXiv190103458F}. It is observed that the entropy level sets appear disconnected for $(+-+)$-bimodal maps (see Figure \ref{fig:plot5}) while they appear connected throughout the adjacent unimodal and $(-+-)$-bimodal regions. Proving the former is the main goal of this paper, and the latter is partially established in \cite[Theorem 1.2]{2019arXiv190103458F}: The restriction of $h_\Bbb{R}$ to the part of Figure \ref{fig:main} which lies below the dotted line 
\begin{equation}\label{the line}
{\rm{Per}}_1(1): \sigma_2=2\sigma_1-3
\end{equation}
is monotonic. This line is dynamically significant as it is where one of the fixed points becomes parabolic. Maps of degree zero below it possess three real fixed points with one of them attracting.  In particular,  the entropy is monotonic throughout the entirety of the $(-+-)$-bimodal region where the dynamics is restricted due to the presence of an attracting fixed point \cite[Lemma 10.1]{MR1246482}.  On the contrary, the fixed points of the $(+-+)$-maps which we shall construct in the next section are repelling (and also not all real). Therefore, their dynamics is \textit{essentially non-polynomial} in the sense of \cite{MR1806289}.

\begin{figure}[ht]
\includegraphics[width=15cm]{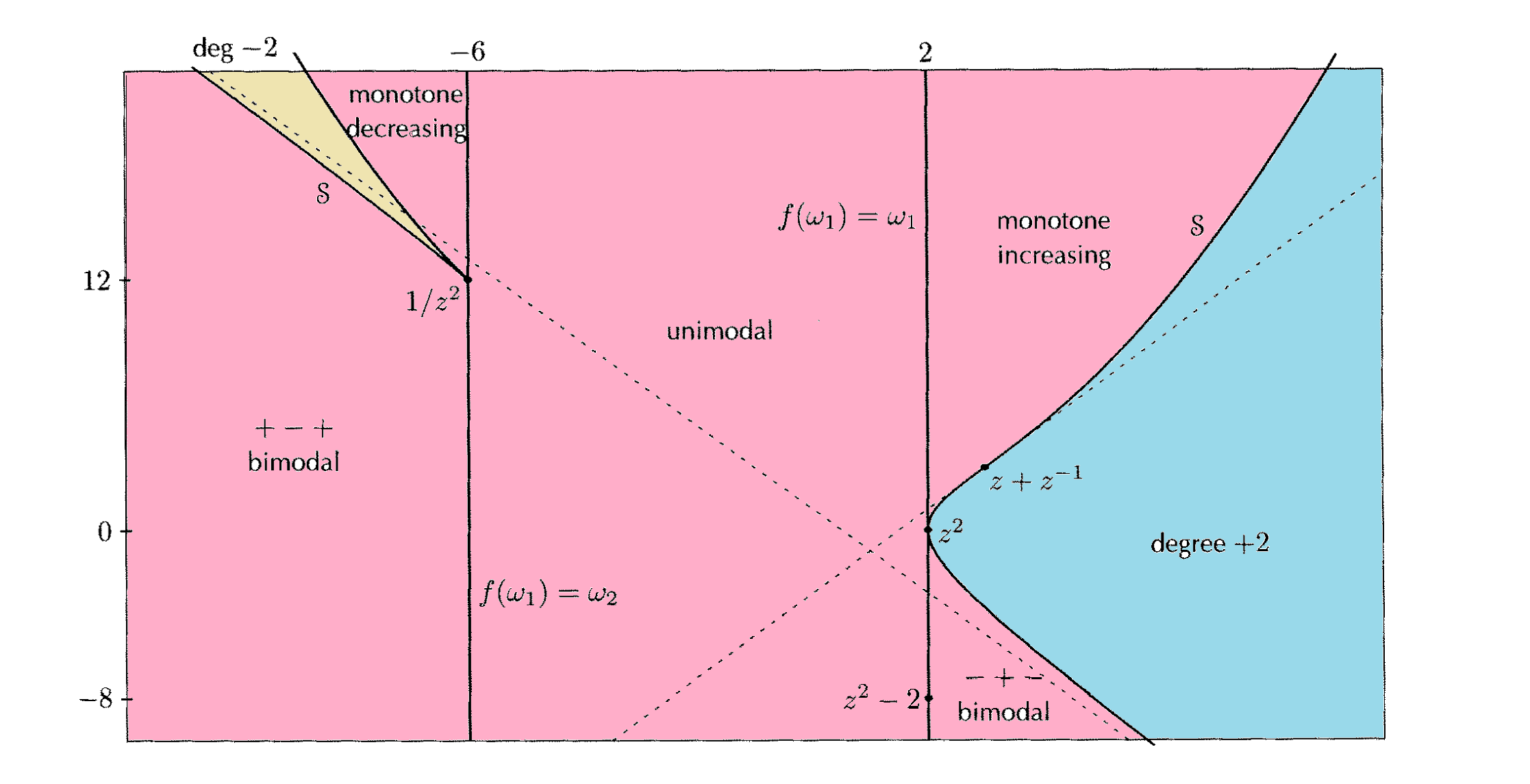}
\caption{A colored version of Figure \ref{fig:main}. The complement in $\mathcal{M}_2(\Bbb{R})$ of the symmetry locus admits three connected components corresponding to possible topological degrees of the restriction
$f\restriction_{\hat{\Bbb{R}}}:\hat{\Bbb{R}}\rightarrow\hat{\Bbb{R}}$ of a quadratic rational map $f$ with real coefficients. If the degree is $\pm 2$, the restriction is a covering map of entropy $\log(2)$. The entropy behavior in the component of degree zero maps (in pink) is far more interesting. }
\label{fig:maincolored}
\end{figure}

\begin{figure}[ht!]
\center
\includegraphics[width=12cm, height=12cm]{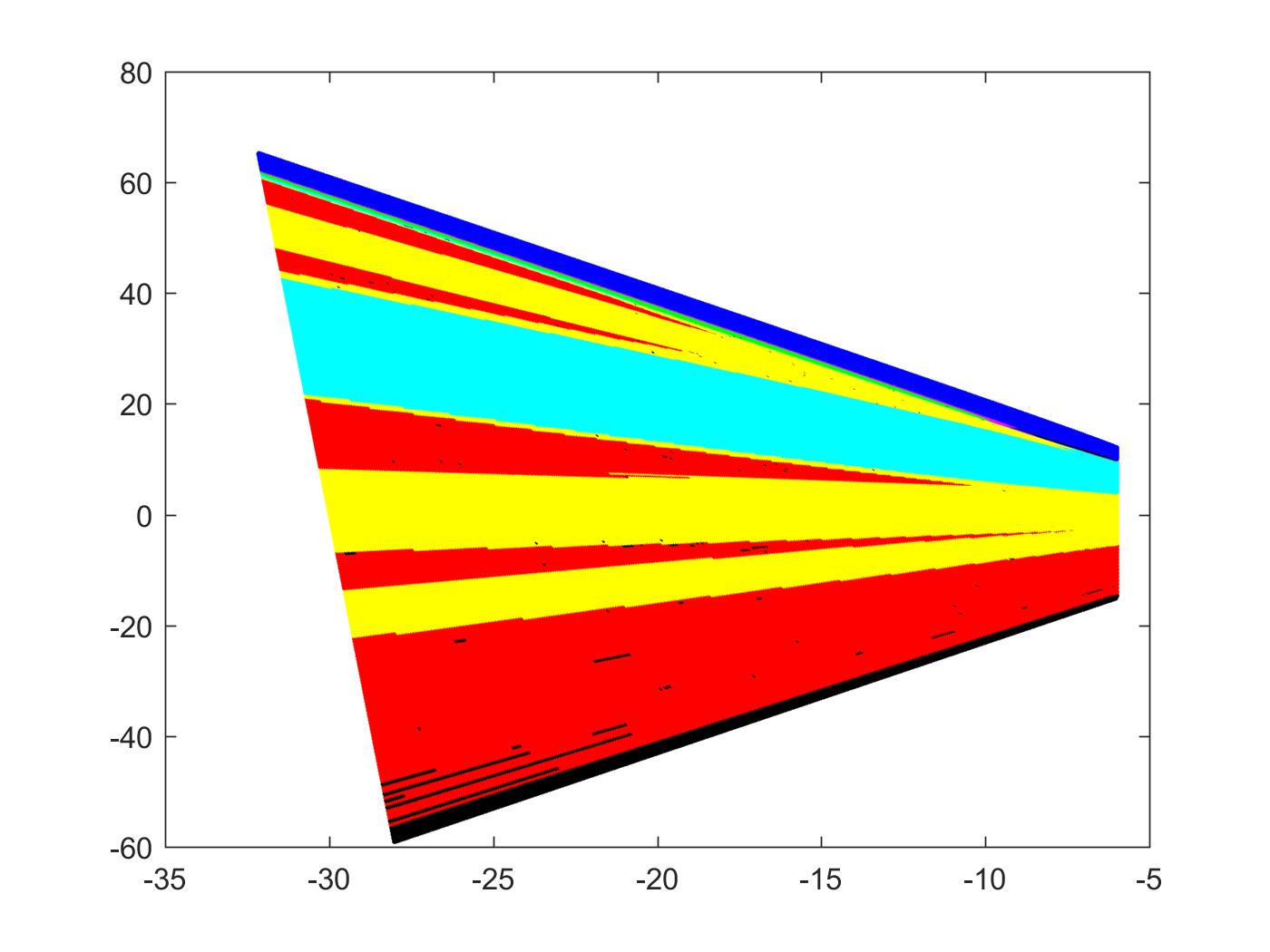}
\caption{An entropy contour plot in the $(+-+)$-bimodal region of the real moduli space (the $(\sigma_1,\sigma_2)$-plane, cf. Figure \ref{fig:main}) adapted from \cite{2019arXiv190103458F}. 
Here the  colors blue, magenta, green, cyan, yellow and red correspond to the entropy being in intervals 
$[0,0.05)$, $[0.05,0.2)$, $[0.2,0.3)$, $[0.3,0.5)$, $[0.5,0.66)$ and $[0.66,\log(2)\approx 0.7]$ respectively. 
The plot is generated utilizing the algorithm introduced in \cite{MR1151977}; and black indicates 
the failure of that algorithm in calculating the entropy. The right vertical boundary line is the post-critical line $\sigma_1=-6$  which intersects the lower skew boundary line ${\rm{Per}}_1(1): \sigma_2=2\sigma_1-3$ (both of them visible in Figure \ref{fig:main}). For $(+-+)$-bimodal maps below this line the Julia set is completely real and the real entropy is $\log(2)$ 
\cite[\S4]{2019arXiv190103458F}. The real entropy tends to zero as we tend to the upper boundary which is part of the symmetry locus.} 
\label{fig:plot5}
\end{figure}

\begin{remark}
The  disconnectedness of the  domain of the entropy function \eqref{real entropy function} suggests that one should phrase the question of monotonicity of $h_\Bbb{R}$ for level sets in just one component of the domain. But as mentioned above, $h_\Bbb{R}\equiv\log(2)$ for maps of degree $\pm 2$; and in Theorem \ref{main} we are dealing with entropy values in 
$\left(0,\log(2)\right)$. Therefore, we  focus on isentropes in the component of degree zero maps hereafter.   
\end{remark}

Before proceeding with a discussion of hyperbolic components in $\mathcal{M}_2(\Bbb{R})$, we point out that symmetric functions \eqref{symmetric functions} of the multipliers  could also be used to describe certain compactifications of moduli spaces $\mathcal{M}_2(\Bbb{R})$ and $\mathcal{M}_2(\Bbb{C})$. We shall need such compactifications in \S\ref{proof}.  
One could compactify $\mathcal{M}_2(\Bbb{C})\cong\Bbb{C}^2$ to the complex projective plane  $\widehat{\mathcal{M}_2(\Bbb{C})}\cong\Bbb{CP}^2$ in which case the added points correspond to degenerate limits of families of quadratic rational maps \cite[\S4]{MR1246482}. To be more precise, notice that if one of the multipliers in \eqref{fixed point formula}, say $\mu_3$, tends to infinity, the product $\mu_1\mu_2$ of the other two tends to $1$; and 
\begin{equation}\label{ratio}
\frac{\sigma_2}{\sigma_1}=\frac{\mu_1\mu_2+\mu_2\mu_3+\mu_3\mu_1}{\mu_1+\mu_2+\mu_3}
=\frac{\frac{\mu_1\mu_2}{\mu_3}+\mu_2+\mu_1}{\frac{\mu_1}{\mu_3}+\frac{\mu_2}{\mu_3}+1}\to\mu_1+\frac{1}{\mu_1}.
\end{equation}
In case that one of $\mu_1$ or $\mu_2$ becomes unbounded too, the other one must tend to $0$ because of \eqref{fixed point formula}; so the limit in \eqref{ratio} would be infinity. 
We conclude that the points at infinity could be thought of as unordered triples $\left\{\mu,\mu^{-1},\infty\right\}$ where $\mu$ belongs to $\hat{\Bbb{C}}:=\Bbb{C}\cup\{\infty\}$.  The sum $\mu+\mu^{-1}\in\hat{\Bbb{C}}$ now serves as a coordinate parameterizing the subset of points at infinity as a copy of the Riemann sphere. The real moduli space $\mathcal{M}_2(\Bbb{R})$ is the real $(\sigma_1,\sigma_2)$-plane where the slope appeared in \eqref{ratio} is real. Hence the closure $\widehat{\mathcal{M}_2(\Bbb{R})}$ of $\mathcal{M}_2(\Bbb{R})\cong\Bbb{R}^2$  in $\widehat{\mathcal{M}_2(\Bbb{C})}\cong\Bbb{CP}^2$ could be identified with the real projective plane $\Bbb{RP}^2$ since it is obtained by adding a copy of $\Bbb{RP}^1$  (the space of lines in $\Bbb{R}^2$) to $\Bbb{R}^2$. Nevertheless, we shall describe another compactification homeomorphic to a disk which is more convenient to work with. For $\sigma_1$ and $\sigma_2$ real, the corresponding limit points would be triples 
$\left\{\mu,\mu^{-1},\infty\right\}$ in which $\mu$ belongs to either $\hat{\Bbb{R}}$, or to the unit circle due to the fact that $\mu+\mu^{-1}$ is required to be real in view of \eqref{ratio}. Replacing $\mu$ with $\mu^{-1}$ if necessary, one may take $\mu$ to be in the interval $[-1,1]$, or on the half-arc $\left\{e^{{\rm{i}}\theta}\mid \theta\in[0,\pi]\right\}$ respectively. The union of these two is a circle which serves as the boundary of a new compactification $\overline{\mathcal{M}_2(\Bbb{R})}$. This is a closed disk fibered above  $\widehat{\mathcal{M}_2(\Bbb{R})}\cong\Bbb{RP}^2$.  We record this discussion for the future usage. 
   
\begin{proposition}\label{compactification}
Adding the boundary circle
\begin{equation}\label{circle at infinity}
{\rm{S}}^1_\infty:=[-1,1]\bigcup\left\{e^{{\rm{i}}\theta}\mid \theta\in[0,\pi]\right\}
\end{equation}
to $\mathcal{M}_2(\Bbb{R})\cong\Bbb{R}^2$ results in a compactification $\overline{\mathcal{M}_2(\Bbb{R})}$ of the real moduli space $\mathcal{M}_2(\Bbb{R})$ homeomorphic to a closed disk with the following topology: The limit in  $\overline{\mathcal{M}_2(\Bbb{R})}$ of the conjugacy classes of a sequence $\left\{g_n\right\}_{n=1}^\infty$ of real quadratic rational maps that degenerate as $n\uparrow\infty$ is a point $\mu\in{\rm{S}}^1_\infty$ provided that there is a sequence 
$\left\{p_n\right\}_{n=1}^\infty\subset\Bbb{CP}^1$ of fixed points with $g'_n(p_n)\to\mu$. 
\end{proposition}

\begin{remark}
The two constituent parts of the circle at infinity \eqref{circle at infinity} correspond to different types of real dynamics. Given a sequence $\left\{g_n\right\}_{n=1}^\infty$ of real quadratic rational maps, when the limit multiplier is $\mu\in (-1,1)$, it means that the other limit multiplier $\mu^{-1}$ is real as well, and hence for $n$ large enough the map $g_n$ has three real fixed points. As $n\uparrow\infty$, one of the corresponding multipliers blows up and the other two tend to the real numbers $\mu$ and $\mu^{-1}$. In contrast, when the maps $g_n$ have a conjugate pair of fixed points and a real fixed point whose multiplier tends to infinity, the limit point on  ${\rm{S}}^1_\infty$ would be a point $e^{{\rm{i}}\theta}$ from the unit circle. The latter is the case for $(+-+)$-bimodal maps which we study in this paper.    
\end{remark}

We finish with a brief treatment of the \textit{hyperbolic components} of the moduli space of quadratic rational maps. Recall that a rational map is called hyperbolic if each of its critical orbits converges to an attracting cycle; equivalent characterizations could be found in \cite[Theorem 19.1]{MR2193309}.   The paper \cite{MR1047139} divides  the hyperbolic components of the 
\textit{critically marked moduli space} (\cite[\S6]{MR1246482}) 
\small
\begin{equation}
\begin{split}
&\mathcal{M}_2^{\rm{cm}}(\Bbb{C}):={\rm{Rat}}_2^{\rm{cm}}(\Bbb{C})\big/{\rm{PSL}}_2(\Bbb{C})\\
&=\left\{(f,c_0,c_1) \mid f \text{ a rational map of degree two with critical points } c_0, c_1\right\}\Bigg/
\begin{matrix}\label{critically marked}
(f,c_0,c_1)\sim\\
\left(\alpha\circ f\circ\alpha^{-1},\alpha(c_0),\alpha(c_1)\right)
\end{matrix}
\end{split}
\end{equation}
\normalsize
into four classes and investigates their topological types.   The corresponding topological types in the unmarked space $\mathcal{M}_2(\Bbb{C})$ can then be deduced \cite[\S7]{MR1246482}. Here, we summarize the four different classes of hyperbolic quadratic maps. 
\begin{itemize}
\item\textbf{Type B: Bitransitive.} Both critical orbits converge to the same attracting periodic orbit but critical points are in immediate basins of different points of this orbit. 
\item\textbf{Type C: Capture.} Only one critical point lies in the immediate basin of an attracting periodic point and the other critical orbit eventually lands there.
\item\textbf{Type D: Disjoint Attractors.} The critical orbits converge to distinct periodic orbits.
\item\textbf{Type E: Escape.} Both critical orbits converge to the same attracting fixed point. This is the only situation where the Julia set of a hyperbolic quadratic rational maps is disconnected (indeed, a Cantor set) \cite[Lemma 8.2]{MR1246482}.
\end{itemize}
There are infinitely many hyperbolic components of types \textbf{B}, \textbf{C} or \textbf{D}. As 
components in $\mathcal{M}_2(\Bbb{C})$, they are topological four cells. The same remains true in the marked space $\mathcal{M}_2^{\rm{cm}}(\Bbb{C})$ with the exception of the component of type \textbf{B} that contains $z\mapsto\frac{1}{z^2}$; the component that we disregard in the proof of Theorem \ref{main}; 
cf. Remark \ref{q=2 irrelevant}.  On the contrary, there is precisely one hyperbolic component of type \textbf{E} which is homeomorphic to $\Bbb{D}\times (\Bbb{C}-\overline{\Bbb{D}})$ \cite[Lemma 8.5]{MR1246482} ($\Bbb{D}$ the open unit disk). 
The escape component is furthermore different in the sense that it is the only hyperbolic component lacking a so-called center -- every other component contains a unique PCF map to which we refer as its center. 

The \textit{real hyperbolic components} obtained from non-empty intersections of complex hyperbolic components with $\mathcal{M}_2(\Bbb{R})$ pertain to our treatment of the real entropy for the following reasons: 
\begin{itemize}
\item the real entropy of a hyperbolic map is the logarithm of an algebraic number and hence $h_\Bbb{R}$ is constant over any real hyperbolic component due to its continuity; 
\item the dynamics on $\hat{\Bbb{R}}$ of the post-critically finite map at the center admits a Markov partition which allows us to  calculate the entropy over the component. 
\end{itemize}
\begin{convention}\label{convention 2}
The intersection of any complex hyperbolic component in $\mathcal{M}_2(\Bbb{C})$  with $\mathcal{M}_2(\Bbb{R})$, if non-vacuous, is connected and contains the center except for the escape component which does not have a center and its intersection with  $\mathcal{M}_2(\Bbb{R})$ has two connected components on which $h_\Bbb{R}$ is either identically zero or identically $\log(2)$ \cite[\S3]{2019arXiv190103458F}. We shall refer to these real escape components as the 
$h_\Bbb{R}\equiv 0$ \textit{escape component} and the $h_\Bbb{R}\equiv \log(2)$ \textit{escape component}. 
Given a complex hyperbolic component $\mathcal{H}\subset\mathcal{M}_2(\Bbb{C})$ different from the escape component, 
by abuse of notation, we show the (connected) real hyperbolic component $\mathcal{H}\cap\mathcal{M}_2(\Bbb{R})$ by 
$\mathcal{H}$ as well. The hyperbolic component in $\mathcal{M}^{\rm{cm}}_2(\Bbb{C})$ to which $\mathcal{H}$ lifts is shown by $\mathcal{H}^\times$.
\end{convention}

 In the next section, we construct a class of bitransitive real hyperbolic components with known entropy values via introducing their centers. 

\section{Constructing Unbounded Hyperbolic Components \texorpdfstring{$\mathcal{H}_{p/q}$}{Hpq}}\label{construction}
The current section is the main technical part of the paper. We first construct a family of hyperbolic PCF quadratic rational maps   in Proposition \ref{the family}  with real coefficients and a comprehensible dynamics on $\hat{\Bbb{R}}$. We then proceed with an analysis of the real hyperbolic components they determine in $\mathcal{M}_2(\Bbb{R})$ and the closure of these components in the compactification $\overline{\mathcal{M}_2(\Bbb{R})}$:  In Propositions \ref{parametrization} and \ref{degeneration} we present a family of curves in the aforementioned components and study their limit points as the maps degenerate.

\begin{convention}\label{setting}
In this section $q$ is an integer larger than one, $p$ belonging to the set $\{1,\dots,q-1\}$  is an integer coprime  to $q$, and $p'$ denotes the multiplicative inverse of $p$ modulo $q$; i.e. the unique element of that set satisfying $pp'\equiv 1\pmod{q}$. 
The  indices  are always considered modulo $q$ and hence are treated as elements of  $\Bbb{Z}/q\Bbb{Z}$.  The critical points are denoted by $c_0$ and $c_1$. 
\end{convention}

The following proposition is the key construction of this paper.

\begin{proposition}\label{the family}
Let $q\geq 2$ be an integer and $p/q\in\Bbb{Q}/\Bbb{Z}$ a fraction in lowest terms. There exists a unique critically finite quadratic rational map $f:=f^\times_{p/q}$ with the following properties:
\begin{enumerate}[\bfseries a.]
\item \label{a} (marked critical points) the two critical points $c_0, c_1$ are labeled;
\item \label{b} (the real condition) $f$ has real coefficients;
\item \label{c} (normalization) $c_0=0$, $f({\rm{i}})={\rm{i}}, f(-{\rm{i}})=-{\rm{i}}$;
\item \label{d} (bitransitive) the critical points $c_0, c_1$ lie in a cycle of period $q$, so the post-critical set ${\rm{P}}_f$ of $f$ is given by $\left\{f^{\circ j}(c_0)\right\}_{j=0}^{q-1}$;
\item \label{e} (rotation number) ${\rm{P}}_f $ is a subset of circle $\hat{\Bbb{R}}=\Bbb{R}\cup\{\infty\}$ equipped with its usual orientation; alternatively,  we may write its elements  as $\{x_0, x_1, \ldots, x_{q-1}\}$ where 
\begin{equation}\label{auxiliary 2}
 0 =c_0=x_0< x_1 < x_2 < \ldots < x_{q-1}<0;
\end{equation}
then we have 
$$  f(x_j)=x_{j+p}$$ 
for all $j \in \Bbb{Z}/q\Bbb{Z}$, so that the restriction $f\restriction_{{\rm{P}}_f}: {\rm{P}}_f \rightarrow {\rm{P}}_f$ forms a cycle with rotation number $p/q$;
\item \label{f} (adjacent critical points) $c_1=x_1$, which in conjunction with the previous property implies 
$f^{\circ p'}(c_0)=c_1$;
\item \label{g} (Markov partition for real dynamics) for $j\in \Bbb{Z}/q\Bbb{Z}$ define intervals in $\hat{\Bbb{R}}$ by 
$$ I_j:=[x_j, x_{j+1}].$$
Then 
\begin{equation}\label{Markov partition}
\begin{split}
&f(I_0)= I_0 \cup I_1 \cup \ldots \cup \underbrace{I_p}_{\text{omitted}} \cup \ldots \cup I_{q-1}\\
&f(I_j)=I_{j+p},\, j=1, \dots, q-1.
\end{split}
\end{equation}
Furthermore, $f\restriction_{I_0}$ reverses orientation while $f\restriction_{I_j},\, j \neq 0$ preserves orientation;

\item \label{h} (entropy) the topological entropy $h_q$ of $f\restriction_{\hat{\Bbb{R}}}$ is the logarithm of the largest positive real root $r_q$ of 
\begin{equation}\label{h_q}
 t^{q-1}-t^{q-2}-\ldots - t^2-t-1=\frac{t^{q}-2t^{q-1}+1}{t-1};
\end{equation}
\item \label{i} (basins and their boundaries) for $j \in \Bbb{Z}/q\Bbb{Z}$ let $\Omega_j$ be the immediate super-attracting basin containing $x_j$; then each $\Omega_j$ is a Jordan domain, and $\partial\Omega_j \supset\{\pm{\rm{i}}\}$;
\item \label{j} (distinguished repelling $q$-cycle) $\partial\Omega_0 \cap \hat{\Bbb{R}}$ consists of two points, with a unique positive real element $\zeta_0>0$, which is periodic of exact period $q$; we denote the remaining elements in this cycle by $\zeta_j \in \partial\Omega_j, \, j=1,\ldots, q-1$, so that $f(\zeta_j)=\zeta_{j+p}$. They are deployed on the circle $\hat{\Bbb{R}}$ as follows  (cf. Figure \ref{fig:p3q10ex}):
\begin{enumerate}[I.]
\item $x_0=0 < \zeta_0<\zeta_1<x_1$,
\item for $j=1,\ldots, p'$: $\zeta_{jp} < x_{jp}$,
\item for $j=1,\ldots, q-p'$: $x_{1+jp}<\zeta_{1+jp}$;
\end{enumerate}

\item \label{k} (normalized basin coordinates) there exist unique Riemann maps (extended to the boundaries)
$\phi_j:(\overline{\Bbb{D}}, 0, 1)\rightarrow(\overline{\Omega_j}, x_j, \zeta_j), j \in \Bbb{Z}/q\Bbb{Z}$.  The first-return map of the immediate basin $\Omega_0$ to itself,  in these coordinates,  is 
$$\phi_j^{-1}\circ f^{\circ q}\circ\phi_j: (\Bbb{D}, 0, 1)\rightarrow(\Bbb{D}, 0, 1);$$
 and is given by $\omega\mapsto\omega^4$.  

\end{enumerate}

\end{proposition}

\begin{figure}
\begin{center}
\includegraphics[width=18cm]{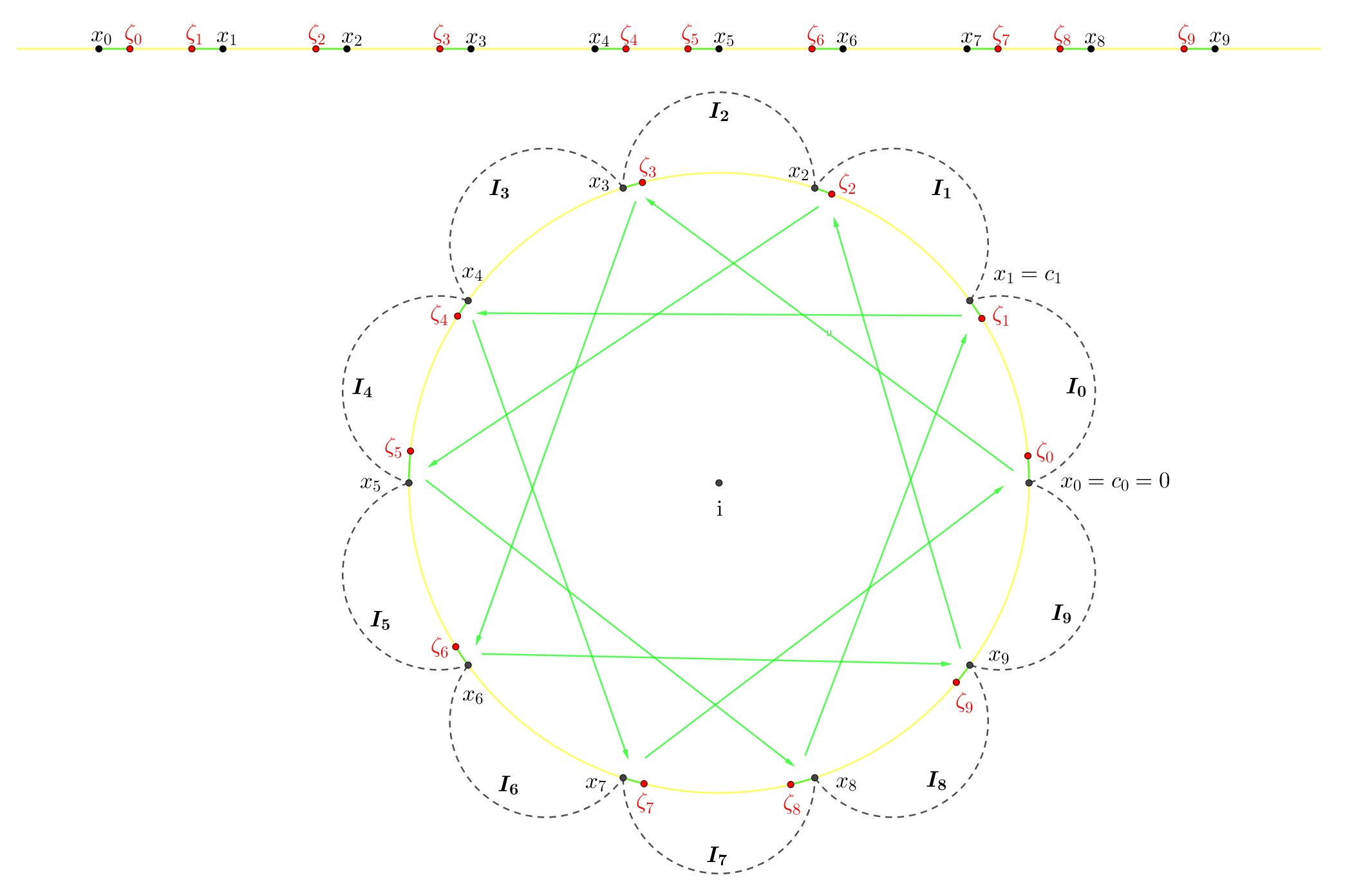}
\end{center}
\caption{An illustration of the real dynamics  described in Proposition \ref{the family} in the case of $q=10, p=3$. Deployment of the post-critical set $\left\{x_j\right\}_{j=0}^{q-1}$ (in black) and the distinguished repelling $q$-cycle $\left\{\zeta_j\right\}_{j=0}^{q-1}$ (in red) is drawn on the circle (at bottom) and on the real line (at top); compare with statement \ref{j} of the proposition. The map $f$ permutes them as  $x_j\mapsto x_{j+p}$ and $\zeta_j\mapsto\zeta_{j+p}$. 
There is a cycle of  open intervals between $\zeta_j$'s and $x_j$'s (in green) which lies in the immediate super-attracting basin of $\left\{x_j\right\}_{j=0}^{q-1}$. The repelling fixed point ${\rm{i}}$ is at the center of the disk, while the repelling fixed point at $-{\rm{i}}$ is the point at infinity in this representation. The Markov partition formed by intervals $I_j=[x_j,x_{j+1}]$ is also visible in the picture.} 
\label{fig:p3q10ex}
\end{figure}

\begin{figure}[t]
\begin{center}
\includegraphics[width=11cm]{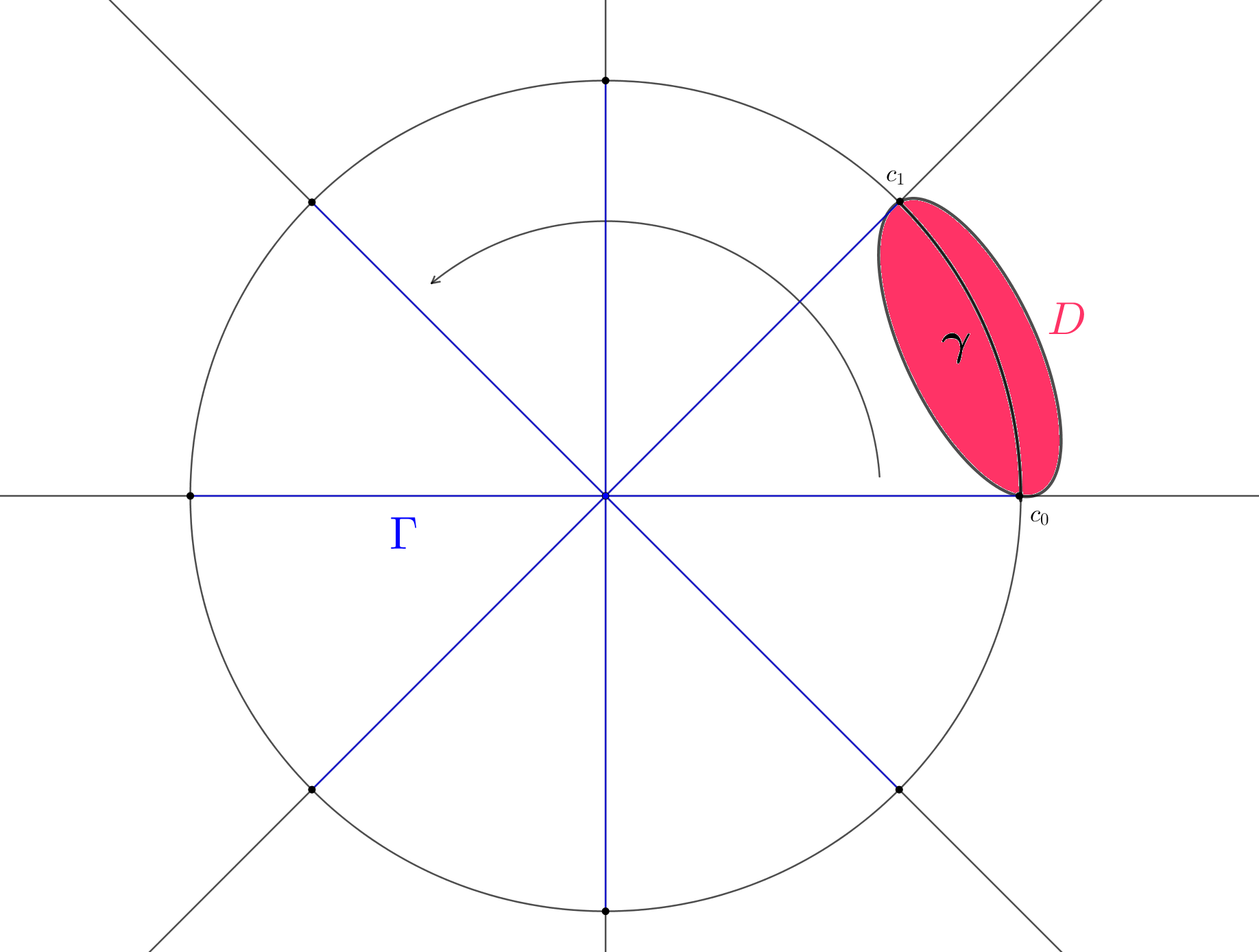}
\end{center}
\caption{The dynamical $w$-plane in the case of $q=8,p=3$. In the proof of Proposition \ref{the family}, one first constructs a rational map $g=g(w)$ which is conjugate to the desired $f_{p/q}(z)$ via \eqref{Cayley}. The construction is by the means of ``blowing up'' a $p/q$-rotation along a curve $\gamma$. In this process, $\gamma$ is slit open (here to the black ellipse) and a topological disk $D$ (here the interior of the ellipse and in red) is then inserted. The endpoints $c_0$ and $c_1$ of $\gamma$ turn out to be the critical points of the resulting rational map $g$ and the center of the rotation a repelling fixed point.  The post-critical set ${\rm{P}}_g$ is the set of black points. The star-shaped curve $\Gamma$ comes up in establishing property \ref{i}.} 
\label{fig:cells}
\end{figure}

\begin{proof}
The proof has three different facets:
\begin{itemize}
\item constructing $f_{p/q}$ as a two-sheeted topological branched cover $S^2\rightarrow S^2$,  proving that it is realized as a rational map via invoking Thurston's characterization of rational maps \cite{MR1251582}, and establishing the uniqueness -- statements \ref{a} through \ref{f};
\item investigating the real dynamics and the corresponding Markov partition -- statements \ref{g} and \ref{h};
\item analyzing the Fatou components -- statements \ref{i} through \ref{k}.
\end{itemize}

\emph{Existence.}  We construct a critically finite topological branched self-cover of the sphere, denoted by $G$, with similar properties, but with the post-critical set in the unit circle, for convenience. We will then apply W. Thurston's combinatorial characterization of rational functions to obtain a rational function $g$. Finally, we conjugate to obtain a real map $f=f_{p/q}$. The specific construction below is a special case of that given in \cite[\S 5.2]{MR1609463}.  \\
\indent
We begin with the change of coordinates 
$$M:\left(\hat{\Bbb{C}}_z, 0, {\rm{i}}, -{\rm{i}}\right)\rightarrow\left(\hat{\Bbb{C}}_w, 1, 0, \infty\right) $$
given by 
\begin{equation}\label{Cayley}
w=M(z)=\frac{{\rm{i}}-z}{{\rm{i}}+z}. 
\end{equation}
We will construct $G:\hat{\Bbb{C}}_w\rightarrow\hat{\Bbb{C}}_w$ first, apply W. Thurston's criterion to obtain a rational map 
$g:\hat{\Bbb{C}}_w\rightarrow\hat{\Bbb{C}}_w$ equivalent to $G$,  and then set $f:=M^{-1}\circ g \circ M$ to obtain our desired map. \\
\indent
Fix $p/q \in \Bbb{Q}/\Bbb{Z}$.  Let ${\rm{P}}:=\left\{{\rm{e}}^{\frac{2\pi {\rm{i}} j}{q}} : j \in \Bbb{Z}/q\Bbb{Z}\right\}$. Equip $\hat{\Bbb{C}}_w$ with the following cell structure: the set of $0$-cells is ${\rm{P}} \cup \{0, \infty\}$; the $1$-cells have two types: sub-arcs of $\{|w|=1\}$ joining ${\rm{e}}^{\frac{2\pi {\rm{i}} j}{q}}$ to ${\rm{e}}^{\frac{2\pi {\rm{i}}(j+1)}{q}}$, and sub-arcs of $\{\arg(w)=2\pi j/q\}$ joining the root of unity it contains to $0$ and $\infty$; see Figure \ref{fig:cells}.  The $2$-cells are then defined as the complementary faces. The order $q$ rotation 
$$ (\hat{\Bbb{C}}_w, {\rm{P}}, 0, \infty)\rightarrow (\hat{\Bbb{C}}_w, {\rm{P}}, 0, \infty)$$
given by 
$$w \mapsto{\rm{e}}^{\frac{2\pi {\rm{i}}p}{q}}w$$
is then a cellular homeomorphism.   We define 
$$ G: (\hat{\Bbb{C}}_w, {\rm{P}}, 0, \infty)\rightarrow (\hat{\Bbb{C}}_w, {\rm{P}}, 0, \infty)$$
by ``blowing up'' (in the sense of \cite[\S 5.2]{MR1609463}) this $p/q$-rotation along the  circular arc $\gamma$ joining $1$ and ${\rm{e}}^{\frac{2\pi {\rm{i}}}{q}}$, which is periodic under the rotation. Here is what this surgery entails. Slit $\hat{\Bbb{C}}_w$ along $\gamma$, and isotop the remainder by pulling the slits apart to form two arcs bounding a topological disk $D$. Map $D$ homeomorphically to the complement of the image of $\gamma$ under 
$w\mapsto{\rm{e}}^{\frac{2\pi {\rm{i}}p}{q}}w$, 
and map $\hat{\Bbb{C}}_w - D$ by pushing the slits back together and then applying the rotation 
$w \mapsto {\rm{e}}^{\frac{2\pi {\rm{i}}p}{q}}w$.  The result is a quadratic map $G$ with post-critical set ${\rm{P}}_G={\rm{P}}$ and two critical points at the endpoints of $\gamma$. We label the critical points of $G$ by distinguishing the one located at the point $w=1$ and calling it $c_0$, and the other $c_1$; see Figure \ref{fig:cells}.\\
\indent
An obstruction to $G$ being equivalent to a rational map, if it exists, takes the form of a non-empty finite collection of pairwise disjoint, pairwise homotopically distinct, essential, non-peripheral curves in the complement of $P$ with certain invariance properties, up to homotopy in $\hat{\Bbb{C}}$. Clearly, any such collection must intersect the circle in the complement of the post-critical set of $G$. However, this is explicitly ruled out by the main theorem in \cite[\S 5.2]{MR1609463}.  We conclude that $G$ is \textit{Thurston equivalent} (conjugate-up-to-isotopy relative to ${\rm{P}}_G$) to a rational function which we denote by $g$. In other words, there is a commutative diagram of the form  
\begin{equation}\label{diagram}
\xymatrix{\hat{\Bbb{C}}_w \ar[r]^\phi \ar[d]_G & \hat{\Bbb{C}}_w \ar[r]^{M^{-1}} \ar[d]_g &\hat{\Bbb{C}}_z \ar[d]_f \\
\hat{\Bbb{C}}_w  \ar[r]_{\phi'} & \hat{\Bbb{C}}_w  \ar[r]_{M^{-1}} &\hat{\Bbb{C}}_z
}
\end{equation}
in which $\phi$ and $\phi'$ are two homeomorphisms which coincide on ${\rm{P}}_G$, and are isotopic relative to ${\rm{P}}_G$. Therefore, the post-critical relations of $G$ carry over to M\"obius conjugate rational maps $f$ and $g$.  Conjugating with a suitable M\"obius transformation, we may assume  that just like $G$ the critical points of $g$ are on the unit circle, $w=1$ is one of the critical points, and the origin is a fixed point.  We can mark the critical points of  
$$f:=M^{-1} \circ g \circ M: \hat{\Bbb{C}}_z\rightarrow\hat{\Bbb{C}}_z$$ 
so that the one corresponding to $w=1$ (i.e. $z=0$) is labeled $c_0$ and the other $c_1$.
We obtain a PCF map $f^\times_{p/q}:=f$ with real critical points and a fixed point at $z={\rm{i}}$. Properties \ref{a} through \ref{f} follow immediately from this construction except property \ref{b} -- the real-ness -- which requires a more careful treatment relying on the \textit{rigidity} part of Thurston's characterization. One needs to verify that $g$ preserves the unit circle or equivalently, $f$ preserves the real circle $\hat{\Bbb{R}}$. Denoting the inversion $w\mapsto\frac{1}{\bar{w}}$ with respect to the unit circle by $u$, $u$ commutes with $G$ by construction, and under the change of coordinates \eqref{Cayley} corresponds to the complex conjugation $z\mapsto\bar{z}$ in the $z$-plane. Therefore, conjugation with $u$ turns diagram \eqref{diagram} into    
\begin{equation*}
\xymatrixcolsep{4pc}\xymatrix{\hat{\Bbb{C}}_w \ar[r]^{u\circ \phi\circ u^{-1}} \ar[d]_{G=u\circ G\circ u^{-1}} & \hat{\Bbb{C}}_w \ar[r]^{M^{-1}} \ar[d]_{u\circ g\circ u^{-1}} &\hat{\Bbb{C}}_z \ar[d]_{\bar{f}} \\
\hat{\Bbb{C}}_w  \ar[r]_{u\circ \phi'\circ u^{-1}} & \hat{\Bbb{C}}_w  \ar[r]_{M^{-1}} &\hat{\Bbb{C}}_z
}
\end{equation*}
where the rational map $u\circ g\circ u^{-1}$ is again Thurston equivalent to $G$, and $\bar{f}:z\mapsto \overline{f(\bar{z})}$ is simply the rational map $f$ with its coefficients being conjugated. By \textit{Thurston's rigidity}, the rational maps $g$ and $u\circ g\circ u^{-1}$ must be M\"obius conjugate. Equivalently, the quadratic rational map $f$ is M\"obius conjugate to $\bar{f}$ which means that the  $(\sigma_1,\sigma_2)$-coordinates of the conjugacy class $\langle f\rangle$ are real. Repeating the argument outlined in \S\ref{background}, it is not hard to see that $f$ should be with real coefficients: We can safely assume that the multiplier of fixed point $z={\rm{i}}$ of $f$, denoted by $\mu$, is real (there exists a fixed point of real multiplier since $\sigma_1,\sigma_2\in\Bbb{R}$).
Then, after an appropriate real M\"obius change of coordinates, $f$ turns into a map of the form \eqref{mixed normal form} with a fixed point of multiplier $\mu$ at $\infty$ and real critical points. 
Formulas \eqref{coordinates} now imply that $f$ is with real coefficients unless it is conjugate to a real map of the form \eqref{mixed normal form-alternative}. But the critical points of the latter map are complex conjugate and this poses extra post-critical relations. Now that we have established $f$ is with real coefficients, the normalization \ref{c} becomes complete: the conjugate of $z={\rm{i}}$ must be a fixed point as well. \\
\indent 
\emph{Uniqueness.} Suppose we have a map $f$ satisfying properties \ref{a} through \ref{f}.  Any such map is conjugate up to isotopy relative to its post-critical set to the combinatorial model $G$.  By Thurston's rigidity theorem  \cite{MR1251582}, any two such maps $f^{(1)}, f^{(2)}$ must be Möbius conjugate.  But the only Möbius conjugation that does not violate properties \ref{a}, \ref{c} and \ref{e} of $f$ is the trivial one: such a transformation must fix the critical point $c_0=0$, and should preserve the
set $\{\pm{\rm{i}}\}$ of fixed points off the real axis as well as the order of real numbers (because of \eqref{auxiliary 2}).\\
\indent
We next turn into parts \ref{g} and \ref{h}. Considering the intervals $I_j=[x_j,x_{j+1}]\, (j\in\{0,\dots,q-1\})$ covering $\hat{\Bbb{R}}$ as in \ref{g}, the critical points $c_0=x_0$ and $c_1=x_1$ occur as the boundary points of the first one $I_0$ (cf. \eqref{auxiliary 2} and property \ref{f}) and hence $f(I_0)$ coincides with the range $f(\hat{\Bbb{R}})$ of $f\restriction_{\hat{\Bbb{R}}}$ (which is not surjective following the discussion in the beginning of \S\ref{background}), and must be one of the closed arcs connecting $f(x_0)=x_p$ to $f(x_1)=x_{p+1}$ on the real circle $\hat{\Bbb{R}}$. But the range must have all points $x_j$ of the critical orbit; so it coincides with the bigger arc
$$
\hat{\Bbb{R}}-{\rm{int}}(I_p)=I_0 \cup I_1 \cup \ldots \cup \underbrace{I_p}_{\text{omitted}} \cup \ldots \cup I_{q-1}
$$
rather than $I_p=[x_p,x_{p+1}]$; that is, $f\restriction_{I_0}$ is orientation-reversing, meaning that $f$ attains a maximum at $x_0=c_0$ and a minimum at $x_1=c_1$. These are the only critical points and thus $f$ is increasing (hence orientation-preserving) on 
$$
[x_1,x_q=x_0]=I_1\cup\dots\cup I_p\cup\dots\cup I_{q-1};
$$  
therefore, 
$f(I_j=[x_j,x_{j+1}])=[f(x_j)=x_{j+p},f(x_{j+1})=x_{j+1+p}]=I_{j+p}.$
Having established \ref{g}, using \eqref{Markov partition} to write the transition matrix for the Markov partition
 $\left\{I_0,\dots,I_p,\dots,I_{q-1}\right\}$ of $f\restriction_{\hat{\Bbb{R}}}:\hat{\Bbb{R}}\rightarrow\hat{\Bbb{R}}$ brings us to the $q\times q$ matrix below:   
\begin{equation}\label{the matrix}
\begingroup
\setlength\arraycolsep{8pt}
\begin{bmatrix}
1&0&0&\cdots&0\\
\vdots&\vdots&\vdots&\vdots&\vdots\\
1&0&0&\cdots&1\\
\mathbf{0}&\mathbf{0}&\mathbf{0}&\mathbf{\cdots}&\mathbf{0}\\
1&1&0&\cdots&0\\
1&0&1&\cdots&0\\
\vdots&\vdots&\vdots&\vdots&\vdots\\
1&0&0&\cdots&0
\end{bmatrix}_{q\times q}.
\endgroup
\end{equation}
Numbering the rows and columns with $0\leq i<q$ and $0\leq j<q$, here the column number $j$ corresponds to $I_j$. The row number $p$ (appeared in bold in \eqref{the matrix}) would be zero since, according to \eqref{Markov partition},  interval $I_p$  is not part of the range. The entropy of the critically finite circle map $f\restriction_{\hat{\Bbb{R}}}$ is the logarithm of the leading eigenvalue of this transition matrix \cite[Theorem 2]{MR1351519}.
Therefore, in order to establish \ref{h}, one needs to compute the characteristic polynomial of \eqref{the matrix}. We claim that it is given by 
$$
t\left(t^{q-1}-t^{q-2}-\dots-t^2-t-1\right)=\frac{t(t^q-2t^{q-1}+1)}{t-1}.
$$
The proof is elementary and will be presented in Lemma \ref{linear algebra} below.\\
\indent
We finally come to the proof of the last three parts which concern the Fatou components of $f$.   The main result of \cite{MR1424402} implies that a critically finite rational map with two critical points which is not conjugate to a polynomial has the property that each Fatou component is a Jordan domain.  Property \ref{k} is now obvious; the Riemann maps extend homeomorphically to the boundary since the basins are Jordan domains. \\
\indent 
Property \ref{j} follows by direct inspection: The intersection of the immediate super-attracting basin $\Omega_j$ of the member $x_j$ of the critical cycle with $\hat{\Bbb{R}}$ contains an open interval around $x_j$ -- the real immediate basin of $x_j$.  In the  case of critical points $x_0=c_0$ and $x_1=c_1$, the real immediate basins are open intervals of the form $\left(\zeta'_0,\zeta_0\right)\ni x_0$  and $\left(\zeta_1,\zeta'_1\right)\ni x_1$ where  $x_0<\zeta_0<\zeta_1<x_1$. These intervals and their boundaries are invariant under the $q^{\rm{th}}$ iterate of $f\restriction_{\hat{\Bbb{R}}}$. The corresponding restrictions (the real first-return maps) are boundary-anchored and unimodal with $x_0$ and $x_1$ as their extrema (cf. property \ref{k}). In view of property \ref{g}, it is not hard to see that $x_0$ is a local minimum and $x_1$ is a local maximum of the $q^{\rm{th}}$ iterate.  Consequently, one must have $f^{\circ q}(\zeta_0)=\zeta_0$ and $f^{\circ q}(\zeta_1)=\zeta_1$, while $f^{\circ q}(\zeta'_0)=\zeta_0$ and $f^{\circ q}(\zeta'_1)=\zeta_1$. Starting to iterate, the periodic endpoints of the real immediate basins $\Omega_j\cap\hat{\Bbb{R}}$ of $x_j$'s form a $q$-cycle (of course a repelling one since these endpoints are Julia) including $\zeta_0$ and $\zeta_1$, whereas the rest of endpoints are wandering, landing at the cycle just described. The deployment of the repelling cycle 
$\left\{\zeta_j:=f^{\circ j}(\zeta_0)\right\}_{j=0}^{q-1}$
with respect to $\left\{x_j\right\}_{j=0}^{q-1}$ could be inferred from $x_0<\zeta_0<\zeta_1<x_1$ in conjunction with the description \ref{g} of the orientation of $f$ on intervals $I_j$: Applying $f$ to $x_0<\zeta_0$ and $\zeta_1<x_1$ yields $\zeta_p<x_p$ and $x_{1+p}<\zeta_{1+p}$ respectively. 
Continuing to apply $f$ repeatedly, we obtain inequalities such as $\zeta_{jp}<x_{jp}$ 
and $x_{1+jp}<\zeta_{1+jp}$ as long as we are not applying $f$ to points from the interval $[x_0,x_1]$ where the function is decreasing. Thus, we have $\zeta_{jp}<x_{jp}$ for $j\in\{1,\dots,p'\}$ and $x_{1+jp}<\zeta_{1+jp}$ for 
$j\in\{1,\dots,q-p'\}$ (notice that $pp'\equiv 1,\, 1+p(q-p')\equiv 0 \pmod{q}$).  \\
\indent
The only remaining property is \ref{i} which will be required for the proof of Proposition \ref{degeneration}. That these basin boundaries contain $\pm{\rm{i}}$ follows from now standard arguments, detailed in \cite[\S 5.4]{MR2691488}.  Here is the idea. Going back to the $w$-plane via \eqref{Cayley}, let $\Gamma_G$ be the star emanating from $0$ in the $1$-skeleton of the cell structure on $\hat{\Bbb{C}}_w$, so that $G: \Gamma_G\rightarrow\Gamma_G$ is a $p/q$ rotation. The equivalence to $g$ yields a star graph $\Gamma$ which is invariant under $g$ up to isotopy relative to ${\rm{P}}_g$.  We may assume $0 \in\Gamma$. In Figure \ref{fig:cells}, $\Gamma$ is just the star of the origin, topologically speaking. There is no critical value on $\Gamma-{\rm{P}}_g$, hence the component of $g^{-1}(\Gamma-{\rm{P}}_g)$
that has the fixed point $w=0=M({\rm{i}})$ is homeomorphic to  $\Gamma-{\rm{P}}_g$.  The closure $\Gamma_1$ of this component is another star-graph isotopic to $\Gamma_0:=\Gamma$ relative to ${\rm{P}}_g \cup \{0,\infty\}$. We may assume that the ``arms'' of $\Gamma_0$ are ``radial'' near the attractors. More precisely, for some $0<\delta_0<1$, viewed in the $\omega$-coordinates of statement \ref{k}, the intersection of each arm of $\Gamma_0$ with the corresponding attracting basin coincides with the locus $\left\{\omega=r{\rm{e}}^{\frac{2\pi{\rm{i}}}{3}}\mid 0\leq r<\delta_0 \right\}$; note that the ``internal angles'' $2\pi/3$ are fixed under the first-return map $\omega\mapsto\omega^4$.  Via backward iteration and the lifting of isotopies, we obtain a sequence $\left\{\Gamma_n\right\}_n$ of mutually isotopic star-graphs centered at the origin and with the post-critical set ${\rm{P}}_g$ as their endpoints. The map $g$ restricts to homeomorphisms $\Gamma_{n+1} \to \Gamma_n$.  The intersections of the $\Gamma_n$ with the attracting basins, in the $\omega$-coordinates, take the form 
$\left\{\omega=r{\rm{e}}^{\frac{2\pi{\rm{i}}}{3}}\mid 0\leq r<\delta_n \right\}.$
  Since $g$ is expanding away from neighborhoods of ${\rm{P}}_g$, as $n\uparrow\infty$, we have $\delta_n \uparrow 1$ (since it must be fixed by the first-return map in the limit); and the $\Gamma_n$'s converge to the union of  internal rays of angle $2\pi/3$ meeting at the origin.   
\end{proof}

Below, is the lemma required for the proof of statement \ref{h} of the proposition:
\begin{lemma}\label{linear algebra}
The characteristic polynomial of matrix \eqref{the matrix} is given by 
$$t\left(t^{q-1}-t^{q-2}-\dots-t^2-t-1\right).$$ 
\end{lemma}

\begin{proof}
Following its description in the proof of Proposition \ref{the family}, matrix  \eqref{the matrix} is the matrix representation of an endomorphism of a vector space with the ordered basis $\left\{v_0,\dots,v_p,\dots,v_{q-1}\right\}$ which is defined as 
\begin{equation}\label{transformation 1}
v_0\mapsto\sum_{k=0}^{q-1}v_k-v_p;\quad   v_j\mapsto v_{j+p},\, j=1, \dots, q-1.
\end{equation}
We first claim that after a linear change of coordinates one may assume that $p=1$. The linear map $v_j\mapsto v_{p'j}$ establishes a conjugacy from \eqref{transformation 1} onto 
\begin{equation}\label{transformation 2}
v_0\mapsto\sum_{k=0}^{q-1}v_k-v_1;\quad   v_j\mapsto v_{j+1},\, j=1, \dots, q-1;
\end{equation}
where, keeping Convention \ref{setting}, $p'$ is a multiplicative inverse of $p$ modulo $q$. 
Therefore, it suffices to compute the characteristic polynomial of transformation \eqref{transformation 2} which has the matrix representation   
$$
A:=\begin{bmatrix}
1&0&0&\cdots& 0&1\\
0&0&0&\cdots&0&0\\
1&1&0&\cdots&0&0\\
1&0&1&\cdots&0&0\\
\vdots&\vdots&\vdots&\ddots&\vdots&\vdots\\
1&0&0&\cdots&1&0
\end{bmatrix}_{q\times q}
=
\begin{bmatrix}
1& \begin{matrix}  0&\cdots& 0 \end{matrix} &1\\
0& \begin{matrix}  0&\cdots& 0 \end{matrix} & 0\\
\begin{matrix} 1\\ \vdots\\ 1 \end{matrix} & \mathlarger{\mathlarger{I_{q-2}}} & \begin{matrix} 0\\ \vdots\\ 0 \end{matrix}
\end{bmatrix}_{q\times q}.
$$
In its first row, the matrix $tI_q-A$ has only two non-zero entries, the initial and the terminal ones. Removing the first row and the first column results in a lower-triangular matrix. Hence the cofactor expansion of determinant with respect to the first row of  $tI_q-A$ yields:
\begin{equation}\label{auxiliary 1}
\begin{split}
\det\left(tI_q-A\right)
&=(t-1)t^{q-1}+(-1)^{q+2}\,\det
\begin{bmatrix}
0&t&\cdots&0&0\\
-1&-1&t&\cdots&0\\
\vdots&\vdots&\ddots&\ddots&\vdots\\
-1&0&\cdots&-1&t\\
-1&0&\cdots&0&-1
\end{bmatrix}_{(q-1)\times (q-1)}\\
&=(t-1)t^{q-1}-t\,\det
\begin{bmatrix}
1&-t&0&\cdots&0\\
1&1&-t&\cdots&0\\
\vdots&\vdots&\ddots&\ddots&\vdots\\
1&0&\cdots&1&-t\\
1&0&\cdots&0&1
\end{bmatrix}_{(q-2)\times (q-2)};  
\end{split}      
\end{equation}
where for the last equality the determinant from the previous line is expanded with respect to its first row. Expanding the determinant appearing in the second line of \eqref{auxiliary 1} down its first column readily results in $1+t+\dots+t^{q-3}$ (the corresponding minors are all upper-triangular). Plugging back in \eqref{auxiliary 1}, we obtain the desired characteristic polynomial as 
$$
(t-1)t^{q-1}-t\left(1+t+\dots+t^{q-3}\right)=t\left(t^{q-1}-t^{q-2}-t^{q-3}-\dots-t-1\right).
$$

\end{proof}

\begin{remark}\label{(+-+)}
As expected, the real quadratic rational maps $f=f^\times_{p/q}$ constructed in Proposition \ref{the family} are bimodal of shape $(+-+)$: Both critical points lie in a periodic orbit, so they belong to the image of 
$f\restriction_{\hat{\Bbb{R}}}:\hat{\Bbb{R}}\rightarrow\hat{\Bbb{R}}$
and the interval map $f\restriction_{f(\hat{\Bbb{R}})}:f(\hat{\Bbb{R}})\rightarrow f(\hat{\Bbb{R}})$
is thus bimodal. The shape is $(+-+)$ since by statement \ref{g} of the proposition, $f$ is decreasing on the interval formed by the critical points $x_0=c_0<x_1=c_1$ while increasing elsewhere. Indeed, given that $(-+-)$-bimodal maps admit attracting (real) fixed points (\cite[Lemma 10.1]{MR1246482}), no real hyperbolic bitransitive map can be $(-+-)$-bimodal. 
\end{remark}

\begin{remark}\label{q=2 irrelevant}
We mostly assume that $q\geq 3$ and disregard the component $\mathcal{H}_{1/2}$. This is due to the fact that this component is not entirely made up of degree zero maps because it intersects the symmetry locus; it contains the class of  $\frac{1}{z^2}$ as its center (cf. Figure \ref{fig:maincolored}), a map which has the non-trivial Möbius automorphism $z\mapsto -z$.  The entropy values $h_q$ of other real hyperbolic components $\mathcal{H}_{p/q}$ is positive once $q>2$ (see Lemma \ref{calculus}); hence  they are away from the symmetry locus and completely contained in the component of degree zero maps in $\mathcal{M}_2(\Bbb{R})$. 
\end{remark}

\begin{convention-proposition}\label{convention 1}
As  maps without labeled critical points, $f_{p/q}$ is M\"obius conjugate to $f_{1-(p/q)}$ -- a symmetry in parameters. Therefore, we assume $p/q\in(0,1/2)$ hereafter ($p/q=1/2$ is excluded by the preceding remark). We denote the corresponding hyperbolic component in the unmarked moduli space by $\mathcal{H}_{p/q}$ which is the image of  either of critically marked components $\mathcal{H}^\times_{p/q}$ or $\mathcal{H}^\times_{1-(p/q)}$.
\end{convention-proposition}
\begin{proof}
There is a real Möbius transformation preserving the set $\{\pm{\rm{i}}\}$ of fixed points, interchanging the critical points $c_0$ and $c_1$, and reversing the order of real line. Conjugating $f^{\times}_{p/q}$ with it results in another quadratic rational map satisfying all properties listed in Proposition \ref{the family} with $q-p$ in place of $p$. The uniqueness part of the proposition implies that this new map coincides with  
$f^{\times}_{(q-p)/q}$. 
\end{proof}

\begin{example}
The real quadratic map $z\mapsto\frac{1}{z^2}$ sends its critical points $0,\infty$ to each other. Conjugating it with an appropriate  M\"obius transformation so  that $c_0=0$ remains a critical point and $\pm{\rm{i}}$ become fixed points as required in statement \ref{c} of Proposition \ref{the family}, we obtain 
$$f^{\times}_{1/2}(z)=\frac{\sqrt{3}\left(\frac{-z+\sqrt{3}}{2z}\right)^2}{\left(\frac{-z+\sqrt{3}}{2z}\right)^2+2}
=\frac{\sqrt{3}(z^2-2\sqrt{3}z+3)}{9z^2-2\sqrt{3}z+3}.$$
\indent
Another example which could be calculated explicitly is the case of denominator $q=3$. The critical points $0,\infty$ of 
$z\mapsto 1-\frac{1}{z^2}$ lie in the orbit $0\mapsto\infty\mapsto 1\mapsto 0$. Again, this turns into $f^{\times}_{1/3}$
after a suitable conjugation. The complex fixed points are the non-real roots $r,\bar{r}$ of
$$
z^3-z^2+1=0.
$$
The transformation $z\mapsto\frac{z}{uz+v}$ where $u=-\Re(r)/\Im(r)$ and $v=|r|^2/\Im(r)$ takes $r$ and $\bar{r}$ to $+{\rm{i}}$ and $-{\rm{i}}$ while fixes $0$. Conjugating $z\mapsto 1-\frac{1}{z^2}$ with it, we arrive at
$$f_{1/3}^{\times}(z)=\frac{v^2z^2-(1-uz)^2}{u(v^2z^2-(1-uz)^2)+v^3z^2}
\approx \frac{1.7753z^2-2.3560z-1}{3.5339z^2+2.7753z+1.1780}.$$
\indent
Similarly, one can obtain 
$$
f_{2/3}^{\times}(z)\approx\frac{1.3876z^2+2.3560z+1}{-2.1914z^2+0.3876z+0.1645}
$$
via normalizing the map $z\mapsto\frac{1}{1-z^2}$ (which is conjugate to the original $z\mapsto 1-\frac{1}{z^2}$) with critical points $0,\infty$ whose critical orbit reads as $0\mapsto 1\mapsto\infty\mapsto 0$. 
\end{example}

The goal for the rest of this section is to study the points at which the components $\mathcal{H}_{p/q}$ touch the boundary of the compactification $\overline{\mathcal{M}_2(\Bbb{R})}$ from Proposition \ref{compactification}. We start with a parametrization of such a component.

Let $p/q\in \Bbb{Q}/\Bbb{Z}$ and $f^\times_{p/q}$ be as in Proposition \ref{the family}. As before, let  $\mathcal{H}^\times_{p/q}$ be the complex hyperbolic component in the moduli space  \eqref{critically marked} of critically marked quadratic rational maps which has $\langle f^\times_{p/q}\rangle$ as its center. 
Suppose $g^{\times}$ is a rational map $g$ along with a labeling of its critical points determining a class 
$\langle g^{\times}\rangle$ in $\mathcal{H}^\times_{p/q}$.  Following Convention \ref{convention 1}, we assume $p/q\neq 1/2$ .  As mentioned in \S\ref{background}, Rees \cite{MR1047139} shows that under this assumption $\mathcal{H}^\times_{p/q}$ is homeomorphic to an open disk and is therefore simply-connected.  As $\langle g^{\times}\rangle$ varies in $\mathcal{H}^\times_{p/q}$, its Julia set moves holomorphically through a motion respecting the dynamics \cite{MR732343}.  Moreover, since elements of  $\mathcal{H}^\times_{p/q}$ are assumed critically marked, this implies that we may consistently label the critical points as $\langle g^{\times}\rangle$ varies in $\mathcal{H}^\times_{p/q}$; and so we may consistently index its attractors $a_j(g)$ along with  immediate basins $\Omega_j(g)$: Just like the setting of Proposition \ref{the family}, the points $\{a_j(g)\}_{j=0}^{q-1}$ constitute an attracting cycle with $\Omega_j(g)$ the immediate basin of $a_j(g)$, the dynamics is given by $a_j(g)\stackrel{g}{\mapsto}a_{j+p}(g)$,  and the critical points of $g$ lie in $\Omega_0(g)$ and $\Omega_1(g)$.  As $\langle g^{\times}\rangle$ is perturbed away from the center $\langle f^{\times}_{p/q}\rangle\in\mathcal{H}^\times_{p/q}$, the distinguished repelling $q$-cycle provided by statement \ref{j} of Proposition \ref{the family} moves holomorphically as well. So we obtain a locally consistently indexed repelling $q$-cycle. Because of the simple-connectedness, this cycle may be globally consistently defined throughout $\mathcal{H}^\times_{p/q}$.  It follows that there exist unique Riemann maps, extending to the boundary, such that  each
\begin{equation}\label{Riemann map}
\phi_j: (\overline{\Bbb{D}}, 0, 1)\rightarrow(\overline{\Omega_j(g)}, a_j(g), \zeta_j(g))
\end{equation}
varies continuously with $g$ ($\Bbb{D}$ denotes the open unit disk).  
We obtain in this way a holomorphic conjugacy 
$$g \circlearrowleft \bigsqcup_{j \in\Bbb{Z}/q\Bbb{Z}} \left(\Omega_j(g), a_j(g), \zeta_j(g)\right)\rightarrow
\Bbb{Z}/q\Bbb{Z}\times (\Bbb{D}, 0, 1)\circlearrowright$$
from the restriction of $g$ to the immediate basin of its attractor to a collection of proper holomorphic maps acting on the disjoint union of $q$ copies of the disk, and permuting the components by adding $p$ to the first coordinate, modulo $q$.\\ 
\indent
Via the normalizations, the composition 
\begin{equation}\label{composition}
\beta_j:=\phi_{j+p}^{-1}(g)\circ g \circ \phi_j(g):(\Bbb{D}, 0, 1)\rightarrow(\Bbb{D}, 0, 1)
\end{equation}
 is either the identity map  -- which is the case of $j \neq 0, 1$ where the map is unramified -- or is a quadratic Blaschke product of the form 
\begin{equation}\label{Blaschke}
\omega\mapsto B_a(\omega):=\frac{1-\overline{a}}{1-a}\cdot \omega\cdot\frac{\omega-a}{1-\overline{a}\omega}, 
a\in\Bbb{D};
\end{equation}
when $j = 0$ or $1$, in which case we are dealing with a degree two self-map of the unit disk since 
$\Omega_0(g)$ and $\Omega_1(g)$ each contains one critical
point of $g$. Let 
\begin{equation}\label{Blaschke space}
\mathcal{B}:=\{(B_a, B_b)\mid a,b\in\Bbb{D}\} \cong \Bbb{D}\times\Bbb{D}
\end{equation}
denote the space of ordered pairs of such normalized quadratic Blaschke products.   The previous paragraph yields a real-analytic map 
$$\mathcal{H}^\times_{p/q}\rightarrow\mathcal{B}$$
given by 
\begin{equation}\label{Blaschke coordinates}
\langle g^{\times}\rangle\mapsto\left(\phi_p^{-1} \circ g \circ \phi_0, \phi_{1+p}^{-1}\circ g \circ \phi_1\right).
\end{equation}
Milnor \cite[Theorem 5.1]{MR2964079} shows that this map is a real-analytic diffeomorphism.  This is where we use maps with marked critical points: Milnor's argument needs $\mathcal{H}^\times_{p/q}$ to be simply-connected to insure that the monodromy of the marked repelling cycle $\left\{\zeta_j(g)\right\}_{j=0}^{q-1}$ is trivial as $\langle g^{\times}\rangle$ varies in $\mathcal{H}^\times_{p/q}$. Below are three simple facts about the parametrization \eqref{Blaschke coordinates}:
\begin{itemize}
\item Pairs of real Blaschke products (i.e. $a,b\in (-1,1)$ in \eqref{Blaschke space}) correspond to real classes in $\mathcal{H}^\times_{p/q}$: just notice that for real maps $g$ the attracting periodic points $a_j(g)$ are real, the basins of attraction $\Omega_j(g)$ are symmetric with respect to the real axis, and  the repelling periodic points 
$\zeta_j(g)$ on the basin boundaries are real too.  Hence \eqref{Blaschke coordinates} furthermore yields a parametrization of the underlying real hyperbolic component. 
\item For a map $g^\times$ given by $(B_a,B_b)$ in the Blaschke coordinate system \eqref{Blaschke coordinates}, the first-return map to the immediate basin $\Omega_0(g)$ containing the first critical point is conjugate to  $B_b\circ B_a:\Bbb{D}\rightarrow\Bbb{D}$. This is due to the fact that the composition $\beta_{(q-1)p}\circ\dots\circ\beta_p\circ\beta_0$ of maps \eqref{composition} (indices modulo $q$ as always) coincides with $\phi_0^{-1}\circ g^{\circ q}\circ\phi_0$ while on the other hand, all constituent parts of the composition are trivial except $\beta_0$ and $\beta_1$ which, according to the definition of the coordinate system, are equal to $B_a$ and $B_b$ respectively. 
\item At the center of the hyperbolic component the corresponding Blaschke products \eqref{Blaschke} are equal to $B_0(\omega)=\omega^2$; this is where  the attracting periodic points $a_0(g)$ and $a_1(g)$ (corresponding to $\omega=0$ under the Riemann map \eqref{Riemann map}) become critical (so $\omega=0$ needs to be a critical point of \eqref{Blaschke}). Notice that the first-return map is then $(\omega\mapsto\omega^2)\circ(\omega\mapsto\omega^2)$; compare with statement \ref{k} of Proposition \ref{the family}. 
\end{itemize}
In view of the facts just mentioned, we shall need the following proposition in order to analyze the first-return map for certain family of real maps in $\mathcal{H}_{p/q}$ to be introduced in the subsequent Proposition \ref{degeneration}.

\begin{proposition}\label{parametrization}
For $0 \leq t<1$, consider the map 
$$ B_t: (\overline{\Bbb{D}}_\omega,0,1) \to (\overline{\Bbb{D}}_\omega,0,1)$$
given by the Blaschke product
$$B_t(\omega)=\omega\cdot\frac{\omega-t}{1-t\omega}$$
as introduced before in \eqref{Blaschke}. 
Let $G_t$ be $B_t \circ B_t: \Bbb{D}_\omega\rightarrow\Bbb{D}_\omega.$
Then 
\begin{enumerate}
\item $B_t(\omega)=-t\omega+O(\omega^2)$ as $\omega\to 0$;
\item $B_t(1)=1$, $B_t'(1)=\frac{2}{1-t}$;
\item for $0<t<1$, denoting the unique critical point of $B_t$ in $\Bbb{D}_\omega$ by $c(t)$,  we have $B_t(c(t))<0<c(t)$;
\item $G_t(w)= t^2\omega+O(\omega^2)$ as $\omega\to 0$;
\item for $0<t<1$,  let $\psi_t: (\Bbb{D}_\omega,0) \to (\Bbb{C}_\xi, 0)$ be the unique holomorphic linearizing map satisfying $\psi_t\circ G_t=t^2\cdot\psi_t$ and $\psi_t(c(t))=1$. Let $X_t:=\left(\Bbb{C}_{\xi}-\{0\}\right)\big/\xi\sim t^2\xi$ be the quotient torus and $\pi_t: \Bbb{C}_\xi-\{0\} \to X_t$  the natural projection.  Let 
$X_t^*:=X_t - \pi_t\circ\psi_t(\{c(t), B_t(c(t))\})$ 
be the quotient punctured torus associated to the attractor at the origin for $G_t$, equipped with its unique hyperbolic metric.  
Then $\gamma_{t\pm}:=\pi_t(\{\pm s\cdot i \mid s>0\})$ are two homotopically distinct closed hyperbolic geodesics on $X_t^*$ whose hyperbolic lengths tend to zero as $t \uparrow 1$.  Under the natural map 
$\Bbb{D}_\omega-\{0\}\rightarrow\Bbb{C}_{\xi}-\{0\}\rightarrow X_t$, these geodesics each have a unique lift whose closure joins the origin to a non-real, repelling fixed point of $B_t \circ B_t$; we denote this pair of complex conjugate repelling fixed points by  
$\rho_\pm(t) \in\partial\Bbb{D}_\omega$;
\item  $G_t'(\rho_{\pm}(t))=\lambda(t):=1+(t-1)(t-3)\to 1$ as $t \uparrow 1$.    
\end{enumerate}
\end{proposition}

\begin{proof}
Properties (1)-(4) follow by direct computation.  In (5), the two lifts near the origin must terminate at a fixed point since they have finite length and are invariant; they have finite length since $G_t$ is uniformly expanding near the circle. To prove (6), first notice that the multipliers of conjugate fixed points $\rho_{\pm}(t)\in\Bbb{D}_\omega$ must be the same since these points are mapped to one another via  $\omega\mapsto\frac{1}{\omega}$, a transformation which commutes with the Blaschke product $B_t$ and hence with $G_t=B_t\circ B_t$. Denoting the multiplier $G_t'(\rho_{\pm}(t))$ by $\lambda(t)$, we apply the holomorphic fixed point formula (\cite[\S12]{MR2193309}) to the degree four rational maps $G_t=G_t(\omega)$. Aside from non-real fixed points $\rho_+(t)$ and $\rho_-(t)$, we already know that $\omega=0$ and $\omega=1$ are also fixed points and of multipliers $t^2$ and 
$\left(\frac{2}{1-t}\right)^2$ respectively. Applying the automorphism  $\omega\mapsto\frac{1}{\omega}$ of $G_t$, $\omega=\infty$ is also a fixed point of multiplier $t^2$. Plugging in the   holomorphic fixed point formula yields
$$
\frac{2}{1-\lambda(t)}+\frac{2}{1-t^2}+ \frac{1}{1-\left(\frac{2}{1-t}\right)^2}=1.  
$$   
Solving for $\lambda(t)$, we obtain $\lambda(t)=1+(t-1)(t-3)$.
\end{proof}

\begin{figure}[t]
\begin{center}
\includegraphics[width=10cm]{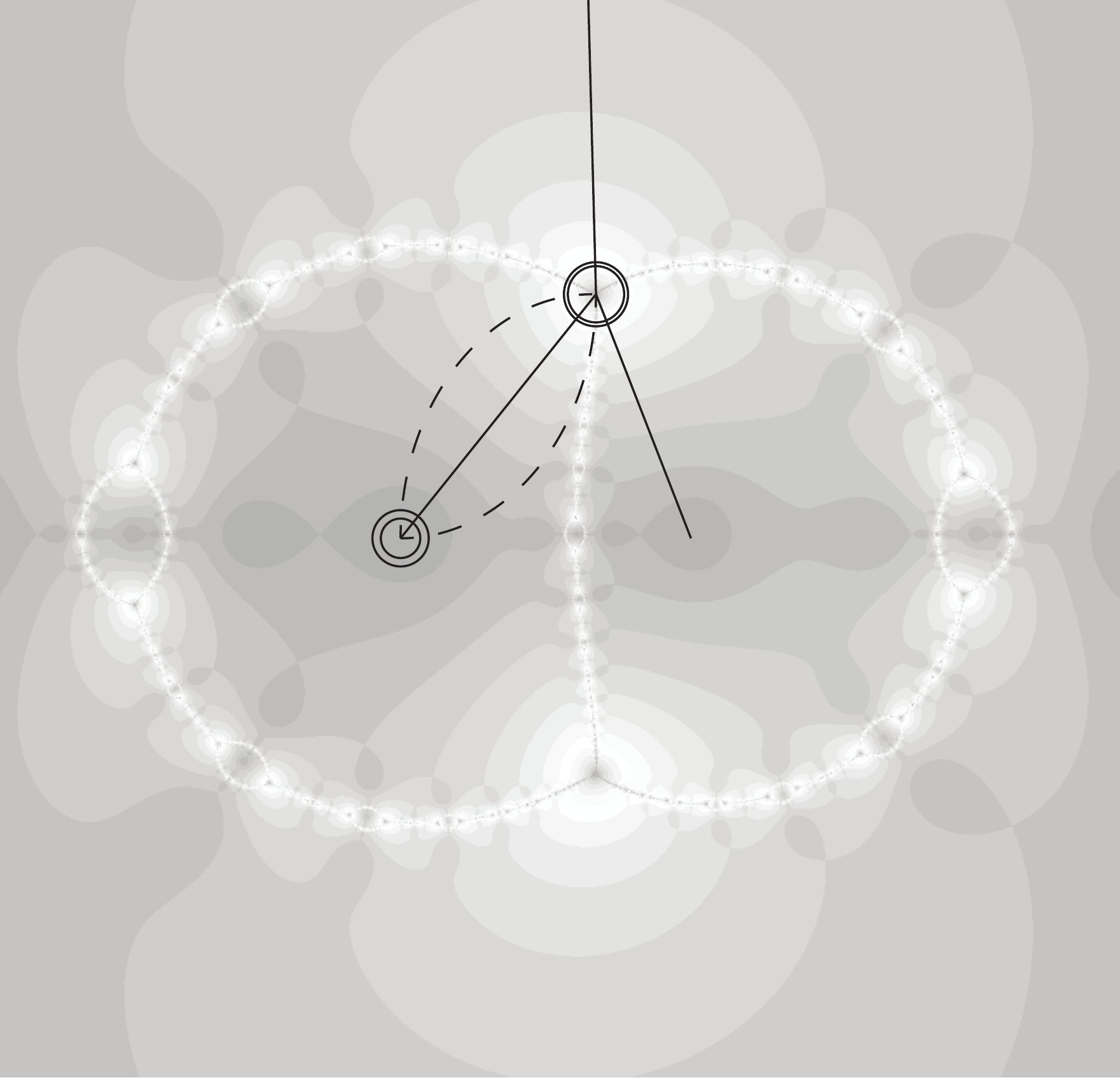}
\end{center}
\caption{Dynamical plane of $f_{p/q,t}$ for $p/q=1/3$ and $t=0.45$. Drawn as circles, not to scale, are fundamental domains for the local return map at the attractor (lower) and a repelling fixed point.  Drawn as solid lines are the three lifts of the geodesics $\gamma_{t,+}$ joining each element of the attracting 3-cycle (normalized here to be $0, 1, \infty$) to the repelling fixed point.  The region between the dashed lines is the lift of an annular neighborhood of $\gamma_{t,+}$ whose modulus tends to infinity as $t \uparrow 1$. This region projects to the quotient torus of the repelling fixed point, forcing its multiplier $\mu_+(t)$ to tend to ${\rm{e}}^{\frac{2\pi {\rm{i}}}{3}}$ as $t \uparrow 1$.  Petersen's estimates utilized in the proof of  Proposition \ref{degeneration} are in terms of the multipliers of repelling fixed points $\lambda(t)$ for the Blaschke product of the return map on the attracting basin.  The situation in the lower-half plane is symmetric, via reflection in the real axis.} 
\label{fig:ply}
\end{figure}

The previous proposition places us in a general situation where so-called ``pinching quasi-conformal deformations'' lead to degenerating families of rational maps. The general idea: on the quotient Riemann surface corresponding to a collection of attracting (or parabolic) basins, one finds a disjoint family of simple closed hyperbolic geodesics, and a corresponding finite collection of lifts (under the natural projection) of these geodesics in the dynamical plane whose closures separate the Julia set. Under these general conditions, pinching the geodesics leads to a family of rational maps that diverges in moduli space, i.e. degenerates; see  e.g. \cite{MR1914001} for details and \cite{MR2691488} for combinatorial conditions that lead to such degenerations and for analogies with Kleinian groups.  \\
\indent
Our situation is combinatorially very simple, and we employ a more direct argument of Peterson \cite[Theorem C]{MR1257034} that leads easily to concrete estimates.  The idea is demonstrated in Figure \ref{fig:ply} and its caption. The lune-shaped region between the dashed lines is a connected component of the lift to the dynamical plane of an annular regular neighborhood $A_t$  of the geodesic corresponding to the imaginary axis on the quotient torus $X_t$ corresponding to the attractor $a_0(t)$. As $t \uparrow 1$, we may choose this neighborhood so that it becomes wider and wider, i.e. $\mod(A_t) \to \infty$.  Examining the endpoint of the lune near the repelling fixed point shows that $A_t$ embeds holomorphically into the quotient torus corresponding to the repelling fixed point common to each $\partial \Omega_j(t)$ at ${\rm{i}}$ (cf. statement \ref{i} of Proposition \ref{the family}).  The core curve of $A_t$ under this embedding is a curve with slope $p/q$. Thus as $t \uparrow 1$ the multiplier at the repelling fixed point must tend to ${\rm{e}}^{\frac{2\pi{\rm{i}}p}{q}}$. Peterson's result gives a refinement of this, based on the observation that the multipliers of  $\rho_\pm(t)$ can be used to give a bigger lower estimate for the conformal width of the lune in the quotient torus corresponding to the fixed point at ${\rm{i}}$.

\begin{proposition}\label{degeneration}

Fix a rational number $p/q$ in $(0,1/2)$. For $0<t<1$, let 
$$\langle f_t^\times:=f^\times_{p/q, t} \rangle\in \mathcal{H}^\times_{p/q}$$ 
be a point of the hyperbolic component which in the  coordinate system  \eqref{Blaschke coordinates} is given by the pair 
$(B_t,B_t)$ of Blaschke product with $B_t$ as in Proposition \ref{parametrization}, and form the corresponding return map $G_t(w)=B_t \circ B_t$. For convenience, keeping to  work with the normalization from part \ref{c} of Proposition \ref{the family}, we  choose a unique normalized representative $f_t^\times$ so that $c_0=0$ and the non-real fixed points are $\pm{\rm{i}}$. 
As usual, denote the underlying unmarked map by $f_t:=f_{p/q,t}$. As in Proposition \ref{parametrization}, let $\lambda(t)$ be the multiplier of the non-real fixed points $\rho_\pm(t)$ of $G_t$.  Let $\mu_{\pm}(t):=f_t'(\pm{\rm{i}})$ be the multipliers of $f_t$ at the non-real fixed points. Then
there exist  logarithms $M_\pm$ of $\mu_\pm$ such that 
$$|M_\pm - 2\pi {\rm{i}} (\pm p/q)| < \log(\lambda(t))=\log(1+(t-1)(t-3));$$
and so 
$$\mu_\pm(t) \to {\rm{e}}^{\pm\frac{2\pi{\rm{i}}p}{q}} \;\text{as}\;  t \uparrow 1.$$
In particular, invoking Proposition \ref{compactification}, the curve 
$\left\{\langle f_t=f_{p/q,t}\rangle\right\}_{t\in [0,1)}$ located inside the hyperbolic component $\mathcal{H}_{p/q}\subset\mathcal{M}_2(\Bbb{R})$ 
starts from $\langle f_{p/q,0}=f_{p/q}\rangle$
and tends to the point 
${\rm{e}}^{\frac{2\pi{\rm{i}}p}{q}}\in{\rm{S}}^1_\infty$ at infinity as $t \uparrow 1$.
\end{proposition}

\begin{proof}
We apply Petersen's estimate \cite[Theorem C]{MR1257034}. We connect our setting with his as follows. The map is $R:=f_{p/q,t}$, the periodic point is fixed and is $\alpha:=\pm{\rm{i}}$ and has rotation number $p/q$, the multiplier at $\alpha$ is $\mu_\pm(t)$, $N:=q$, and what Petersen terms the ``conjugate multipliers'' $\lambda_1=\dots=\lambda_N$ are the multiplier $\lambda(t)$ at the fixed points $\rho_\pm(t)$ of  the Blaschke products $G_t$. Converted to our setting, his theorem says 
$$|M_\pm - 2\pi {\rm{i}} (\pm p/q)| \leq B \cdot 2\sin\theta \cdot \frac{\log \lambda(t)}{q^2N}$$
where $0<B\leq 1$ and $\theta$ is the angle between $2\pi {\rm{i}}$ and $M-2\pi {\rm{i}} p/q$. Applying the trivial bounds $B \leq 1$,  $|\sin\theta| \leq 1$, and $q \geq 2$ yields the estimate. In the holomorphic fixed point formula \eqref{fixed point formula} if two of the multipliers (here $\mu_{\pm}$) tend to reciprocal numbers, the third one (here the multiplier of the real fixed point) must blow up.  Hence the curve 
$\left\{\langle f_t=f_{p/q,t}\rangle\right\}_{t\in [0,1)}\subset\mathcal{M}_2(\Bbb{R})$ (and thus the real hyperbolic component $\mathcal{H}_{p/q}$)
is unbounded in $\mathcal{M}_2(\Bbb{R})$, and as $t \uparrow 1$ tends to the ideal point 
${\rm{e}}^{\frac{2\pi{\rm{i}}p}{q}}$ on the boundary of the compactification $\overline{\mathcal{M}_2(\Bbb{R})}$.
\end{proof}

\begin{figure}[ht!]
\center
\includegraphics[width=10cm, height=10cm]{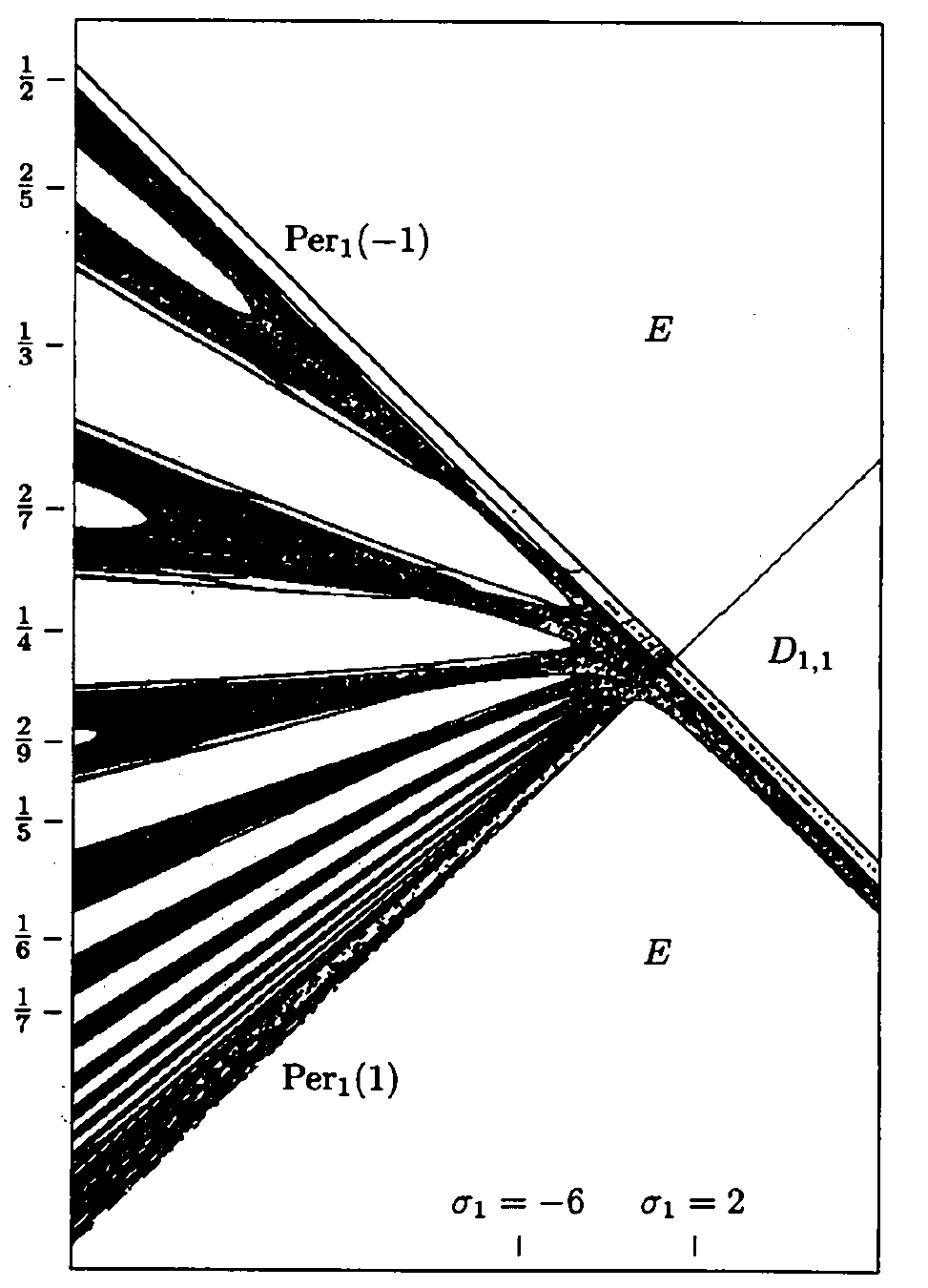}
\caption{An illustration of real bitransitive hyperbolic components  adapted from  \cite[Figure 17]{MR1246482}. Compare the limit points with Remark \ref{cor:singleton}.}
\label{fig:components}
\end{figure}

\begin{remark}\label{cor:singleton}
By Proposition \ref{degeneration}, the closure of the hyperbolic component $\mathcal{H}_{p/q}\subset\mathcal{M}_2(\Bbb{R})$  in $\overline{\mathcal{M}_2(\Bbb{R})}$
meets ${\rm{S}}^1_\infty$ at the point ${\rm{e}}^{\frac{2\pi{\rm{i}}p}{q}}$. It is known that the only possible limit points of $\mathcal{H}_{p/q}$ at infinity correspond to multipliers ${\rm{e}}^{\frac{2\pi{\rm{i}}r}{s}}$ where $s \leq q$ \cite[Proposition 1]{MR1764925}.
\end{remark}

\section{Proof of the Main Theorem}\label{proof}
We prove  Theorem \ref{main} in this section by exploiting the results of \S\ref{construction}. The key observation regarding the hyperbolic components $\mathcal{H}_{p/q}$ constructed in that section is the independence of the real entropy value attained over them from the numerator $p$. This existence of disjoint real hyperbolic components of the same real entropy is an essential insight in the core of the proof. 

\begin{proof}[Proof of Theorem \ref{main}]
We first introduce the number $h'$. Recall that the entropy values $h_q$ have appeared in part \ref{h} of Proposition \ref{the family}: For $q>2$, the value $h_q$ of the real entropy function over the hyperbolic components $\mathcal{H}_{p/q}$ is the logarithm of the largest root of the polynomial 
$P_q(t):=t^q-2t^{q-1}+1$ in $(1,+\infty)$.
Form the  sequence $\left\{h_q\right\}_{q=3}^\infty$. A calculus argument indicates that  $\left\{h_q\right\}_{q=3}^\infty$  is a strictly increasing sequence of positive numbers converging to $\log(2)$; this is the content of Lemma \ref{calculus} below.  We set $h'$ to be $h_3$. \\
\indent
Next, we shall prove that if the integer $q>3$ is large enough, say $q>q_0$, for any entropy value $h\in\left(h_3,h_q\right]$  the isentrope $h_\Bbb{R}=h$ is disconnected. This would finish the proof since the union of these intervals over $q\in\left\{q_0+1,q_0+2,\dots\right\}$  coincides with $\left(h',\log(2)\right)$.\\
\indent
Consider the compactification $\overline{\mathcal{M}_2(\Bbb{R})}\cong\overline{\Bbb{D}}$ of the real moduli space introduced in Proposition \ref{compactification} which was formed by adding the circle at infinity \eqref{circle at infinity}.  A portion of $\overline{\mathcal{M}_2(\Bbb{R})}$ encompassing the region of $(+-+)$-bimodal maps is illustrated in Figure \ref{fig:compactification}. The right boundary of this region is the vertical line $\sigma_1=-6$ which is the locus where one of the critical points is mapped to another; cf. Figure \ref{fig:main}. The real entropy  is monotonic along this line, and decreases as the other coordinate $\sigma_2$ increases; see Lemma \ref{monotone along -6} and Figure \ref{fig:graph} below. In the construction of Proposition \ref{the family}, note that $f_{p/q}$ belongs to this line if $p=1$.
In particular, the classes $\left\langle f_{1/3}\right\rangle$ and $\left\langle f_{1/q}\right\rangle$ lie on this line. In contrast, a point such as 
$\left\langle f_{p/q}\right\rangle$ where $p/q\in(1/3,1/2)$ does not belong to the line $\sigma_1=-6$ as the numerator is different from $1$; cf. Proposition \ref{the family}, statement \ref{f}. Such a numerator $p$ exists; in fact, if $q$ is larger than $q_0:=12$ we can pick $p$ so that 
$$1/q< 1/3<p/q<1/2;$$ 
see Lemma \ref{number-theoretic} below. (Keep in mind that in working with unmarked components $\mathcal{H}_{p/q}$ we take the fraction to be in $(0,1/2)$; cf. Convention \ref{convention 1}.)
Fixing such a $p$, we consider the disjoint curves 
$\left\{\langle f_{1/q,t}\rangle\right\}_{t\in [0,1)}$, $\left\{\langle f_{p/q,t}\rangle\right\}_{t\in [0,1)}$ and 
$\left\{\langle f_{1/3,t}\rangle\right\}_{t\in [0,1)}$ in the corresponding hyperbolic components. The entropy is constant along these curves with $h_\Bbb{R}\equiv h_q$ along the former two and $h_{\Bbb{R}}\equiv h_3$ along the third one.  They start at the points 
$\langle f_{1/q}\rangle$, $\langle f_{p/q}\rangle$ and $\langle f_{1/3}\rangle$ and, by Proposition \ref{degeneration}, as $t\uparrow 1$ tend to limit points ${\rm{e}}^{\frac{2\pi{\rm{i}}}{q}}$, ${\rm{e}}^{\frac{2\pi{\rm{i}}p}{q}}$ and ${\rm{e}}^{\frac{2\pi{\rm{i}}}{3}}$ respectively. Now form the subset $L$ of  $\mathcal{M}_2(\Bbb{R})$ comprising of the part of the line $\sigma_1=-6$ that lies above $\langle f_{1/3}\rangle$ along with the curve  
$\left\{\langle f_{1/3,t}\rangle\right\}_{t\in [0,1)}$. Due to the monotonicity along $\sigma_1=-6$, the values that $h_{\Bbb{R}}$ attains on $L$ belong to $\left[0,h'=h_3\right]$; thus $L$ does not intersect the isentrope $h_\Bbb{R}=h$. The closure of $L$ in the closed disk $\overline{\mathcal{M}_2(\Bbb{R})}$ touches the boundary circle at two points; and its complement in $\mathcal{M}_2(\Bbb{R})$ admits two connected components; see Figure \ref{fig:compactification}. The isentrope $h_\Bbb{R}=h$ is contained in this complement, and it suffices to show it has points in both components. The curves 
$\left\{\langle f_{1/q,t}\rangle\right\}_{t\in [0,1)}$ and $\left\{\langle f_{p/q,t}\rangle\right\}_{t\in [0,1)}$ lie in different connected components of $\mathcal{M}_2(\Bbb{R})-L$ due to the fact that, according to the way $p$ was chosen, the corresponding limit points ${\rm{e}}^{\frac{2\pi{\rm{i}}}{q}}$ and ${\rm{e}}^{\frac{2\pi{\rm{i}}p}{q}}$ of them on the boundary circle are located on different sides of the limit point ${\rm{e}}^{\frac{2\pi{\rm{i}}}{3}}$ of 
$\left\{\langle f_{1/3,t}\rangle\right\}_{t\in [0,1)}$.
Connect points from the former two curves to the latter curve via line segments  arbitrarily (as in Figure \ref{fig:compactification}). The interiors of the segments are in different components of $\mathcal{M}_2(\Bbb{R})-L$, and the  value $h\in\left(h',h_q\right]$ is realized by $h_\Bbb{R}$ on them by a simple application of the intermediate value theorem. This concludes the proof.

\begin{figure}[ht!]
\center
\includegraphics[height=10cm,width=11cm]{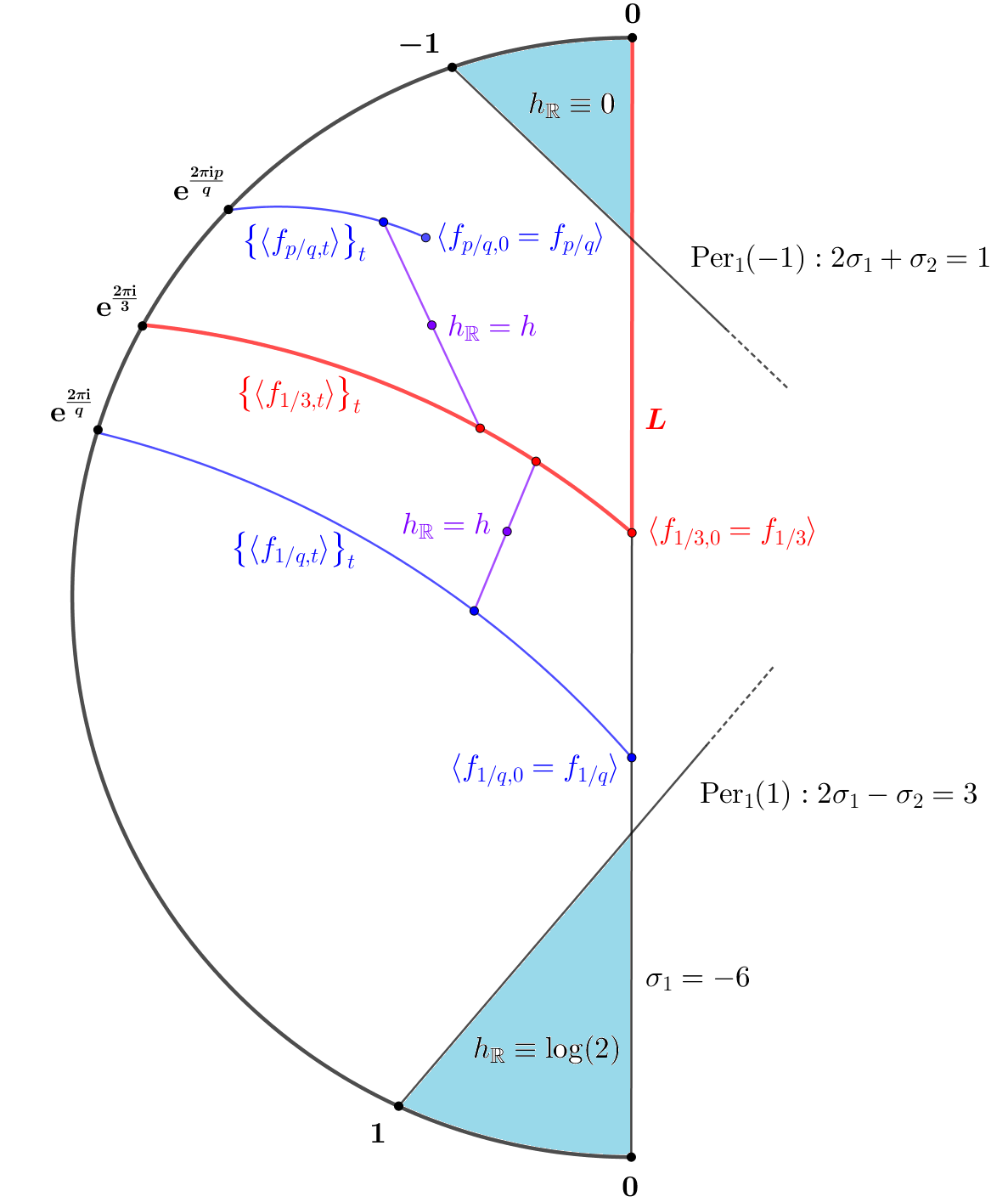}
\caption{The portion of the compactification $\overline{\mathcal{M}_2(\Bbb{R})}$ of $\mathcal{M}_2(\Bbb{R})$ which is located to the left of the post-critical line $\sigma_1=-6$; compare with Figure \ref{fig:main}. The boundary circle \eqref{circle at infinity} (in thick black) and the ideal points on it (in black bold font) are visible.  The lines ${\rm{Per}}_1(\pm 1)$ are the loci where one of the real fixed points becomes parabolic (hence of multiplier $+1$ or $-1$); and the colored regions cut by them lie in the escape components. Between these two lines we have $(+-+)$-bimodal maps and certain curves  relevant to the proof of Theorem \ref{main}. Each entropy value $h\in\left(h_3,h_q\right]$ is realized on the purple segments but not on the broken red curve $L$, hence the disconnectedness of the level set $h_\Bbb{R}=h$.}
\label{fig:compactification}
\end{figure}

\end{proof}

We conclude the paper with the lemmas alluded to during the proof.
\begin{lemma}\label{calculus}
For any $q\geq 3$ the polynomial 
\begin{equation}\label{polynomial}
P_q(t)=t^q-2t^{q-1}+1
\end{equation}
has a unique root $r_q$ in the interval $(1,2)$. This is  the largest positive root of $P_q$ and moreover, $\left\{r_q\right\}_{q=3}^\infty$ is an strictly increasing sequence tending to $2$.
\end{lemma}

\begin{proof}
First notice that $P'_q(t)=qt^{q-1}-2(q-1)t^{q-2}=t^{q-2}\left(qt-2(q-1)\right)$; so $c_q:=\frac{2(q-1)}{q}>1$ is the only non-zero critical point.  The function $P_q$ varies as 
\begin{center}
\begin{tabular}{c|ccccccccccc}
$t$ & $-\infty$ & & $0$ & & $1$ & & $c_q$ & & $2$ & & $+\infty$\\
\hline
$P_q(t)$ & $-\infty$ & $\nearrow$ & $1$ & $\searrow$  & $0$ & $\searrow$ & $P_q(c_q)$ & $\nearrow$ & $1$ & $\nearrow$ & $+\infty$ 
\end{tabular}
\end{center}
when $q$ is odd, and as 
\begin{center}
\begin{tabular}{c|ccccccccccc}
$t$ & $-\infty$ & & $0$ & & $1$ & & $c_q$ & & $2$ & & $+\infty$\\
\hline
$P_q(t)$ & $+\infty$ & $\searrow$ & $1$ & $\searrow$  & $0$ & $\searrow$ & $P_q(c_q)$ & $\nearrow$ & $1$ & $\nearrow$ & $+\infty$ 
\end{tabular}
\end{center}
for $q$ even. We deduce that $P_q(c_q)<0$ and $P_q$ has a unique root $r_q$ different from $1$ which lies in 
$\left(c_q=\frac{2(q-1)}{q},2\right)\subset (1,2)$. We deduce that $\lim_{q\to\infty}r_q=2$. The only thing left to verify is the inequality $r_q<r_{q+1}$. Assume otherwise: $r_q\geq r_{q+1}$. But then $r_q\geq r_{q+1}>c_{q+1}>c_q$, and we know that $P_q$ is strictly increasing on $(c_q,+\infty)$.  Thus
$$
(r_{q+1})^q-2(r_{q+1})^{q-1}+1=P_q(r_{q+1})\leq P_q(r_q)=0.
$$
Multiplying both sides of the inequality above by $r_{q+1}$, we get
$$
(r_{q+1})^{q+1}-2(r_{q+1})^q+r_{q+1}\leq 0.
$$
On the other hand, $P_{q+1}(r_{q+1})=(r_{q+1})^{q+1}-2(r_{q+1})^q+1=0$. It follows immediately that $r_{q+1}\leq 1$, a contradiction to $r_{q+1}\in (1,2)$.
\end{proof}

\begin{lemma}\label{monotone along -6}
Restricted to the line $\sigma_1=-6$ -- characterized by the post-critical condition $c_0\mapsto c_1$ -- of the real moduli space, the real entropy is a decreasing function of the other coordinate $\sigma_2$. (The corresponding real entropy graph and the complex bifurcation locus could be found in Figure \ref{fig:graph} and \cite[Figure 13]{MR1246482}.)
\end{lemma}

\begin{proof}
The main idea is to parametrize the conjugacy classes along $\sigma_1=-6$ by representatives of the form $a\cdot f_0$, where $f_0$ is an appropriate map and $a \in \Bbb{R}$ lies in some suitable interval; and then use the literature on the monotonicity of entropy for  one-dimensional families in this form. Up to conjugacy, real quadratic rational maps with the post-critical condition $c_0\mapsto c_1$ are in the form of  $b+\frac{1}{z^2}$ where $b\in\Bbb{R}$. Invoking the formulas derived in \cite[Appendix C]{MR1246482}, the $(\sigma_1,\sigma_2)$-coordinates of the corresponding points of $\mathcal{M}_2(\Bbb{R})$ are given by $\sigma_1=-6$ and $\sigma_2=4b^3+12$. For $b\geq 0$ the topological entropy of 
$x\in\hat{\Bbb{R}}\mapsto b+\frac{1}{x^2}\in\hat{\Bbb{R}}$ is zero due to the fact that the points in $(-\infty,0)$ are wandering and the restriction to the invariant interval $[0,+\infty]$ is monotone decreasing. This covers the upper ray starting at 
$\left\langle\frac{1}{z^2}\right\rangle=(-6,12)$; cf. Figure \ref{fig:main}. At points
$\left\langle b+\frac{1}{z^2}\right\rangle$  with $b<0$ -- which are below $\left\langle\frac{1}{z^2}\right\rangle$ -- 
the function $h_\Bbb{R}$  eventually becomes positive because the maps belong to the $h_\Bbb{R}\equiv\log(2)$ real escape component provided that $b\ll 0$.  The linear change of coordinates $x\mapsto\sqrt{-b}x$ results in the conjugate maps $z\mapsto a\left(-1+\frac{1}{x^2}\right)$ where 
$a:=\left(\sqrt{-b}\right)^3$. For $a>0$ the image of the map 
$$x\in\hat{\Bbb{R}}\mapsto a\left(-1+\frac{1}{x^2}\right)\in\hat{\Bbb{R}}$$
is the interval $[-a,+\infty]$ which contains only one critical point, $x=0$. Therefore, one needs to establish the monotonicity of entropy for the family 
$$
\left\{x\mapsto a\left(-1+\frac{1}{x^2}\right):[-a,+\infty]\rightarrow[-a,+\infty]\right\}_{a>0}
$$
of unimodal interval maps. This immediately follows from \cite[Theorem 7.2]{2019arXiv190206732L}.
\end{proof}

\begin{figure}[ht!]
\center
\includegraphics[width=10cm]{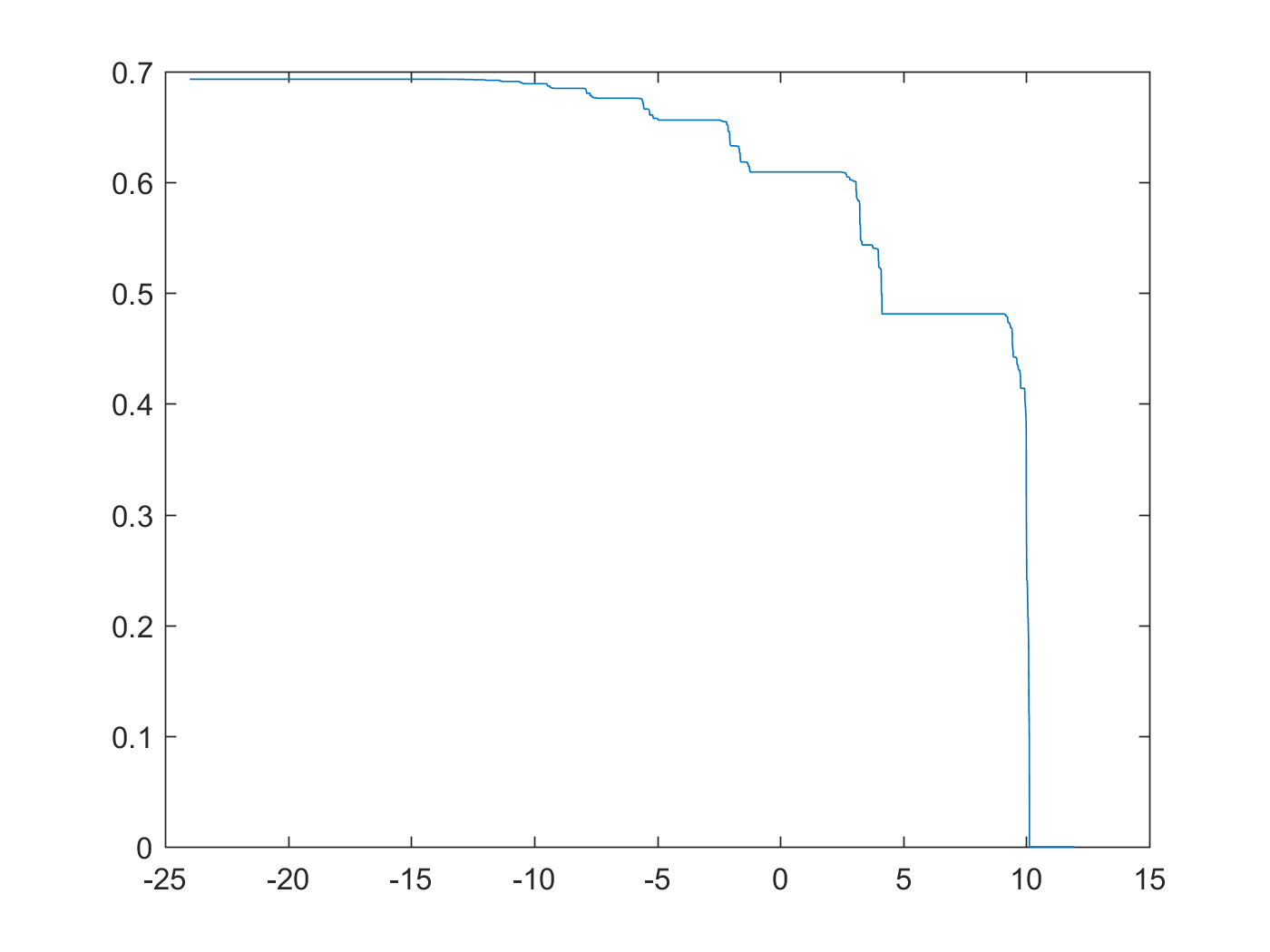}
\caption{A graph of the real entropy along the post-critical line $\sigma_1=-6$ versus the coordinate $\sigma_2$. The topological entropy has been calculated via the algorithm developed in \cite{MR1002478} for unimodal maps. The entropy is a decreasing function of $\sigma_2$; cf. Lemma \ref{monotone along -6}.}
\label{fig:graph}
\end{figure}

\begin{remark}
The line $\sigma_1=-6$ is dynamically significant because it is the locus where $c_0 \mapsto c_1$. By contrast, the entropy fails to be monotonic along arbitrary vertical lines $\sigma_1=a$ with $a\ll 0$ since such a line intersects disjoint unbounded hyperbolic components of the form $\mathcal{H}_{1/q}$ and  $\mathcal{H}_{p/q}$ which are of the same real entropy value.
\end{remark}

\begin{lemma}\label{number-theoretic}
For any $q>12$  there exists an integer $p$ coprime to $q$ satisfying 
$$1/3<p/q< 1/2.$$
\end{lemma}

\begin{proof}
The  desired integer $p$ must be larger than $q/3$ and smaller than $q/2$.   For $q$ odd one can set $p=(q-1)/2$ which is larger than $q/3$ provided that $q>3$. If $q\equiv 2 \pmod{4}$, the integer $p:=q/2-2$ is coprime to $q$. It is larger than $q/3$ provided that $q>12$. Finally, in the case of  $q\equiv 0 \pmod{4}$, one can choose 
$p:=q/2-1$ which satisfies  $\gcd(p,q)=1$, and also $p>q/3$ when $q>6$.
\end{proof}

\bibliographystyle{alpha}
\bibliography{bibcomprehensive}

\end{document}